\documentclass[12pt, leqno]{article}
\usepackage{amsmath,amsfonts,amssymb,wasysym,mathrsfs}
\usepackage[latin1]{inputenc}
\usepackage{shapepar}
\usepackage{graphicx}
\usepackage{cancel}
\usepackage{float}
\usepackage{hyperref}
\usepackage{ifpdf} 
\usepackage{color}
\newcommand{\R}{{\mathbb R}}
\newcommand{\ren}{{\mathbb R}^N}

\newcommand{\be}[1]{\begin{equation}\label{#1}}
\newcommand{\ee}{\end{equation}}

\newcommand{\prf}{\par\smallskip\noindent{\sl Proof. \/}}
\newcommand{\finprf}{\unskip\null\hfill$\;\square$\vskip 0.3cm}

\newenvironment{proof}{\prf}{\finprf}
\newtheorem{theorem}{Theorem}[section]
\newtheorem{lemma}{Lemma}[section]
\newtheorem{corollary}[theorem]{Corollary}
\newtheorem{proposition}[theorem]{Proposition}
\newtheorem{remark}[theorem]{Remark}

\newcommand{\ve}{\varepsilon}

\numberwithin{equation}{section}
\newcommand{\nc}{\normalcolor}

\def\qed{\,\unskip\kern 6pt \penalty 500
\raise -2pt\hbox{\vrule \vbox to8pt{\hrule width 6pt
\vfill\hrule}\vrule}\par}
\definecolor{darkblue}{rgb}{0.05, .05, .65}
\definecolor{darkgreen}{rgb}{0.1, .65, .1}
\definecolor{darkred}{rgb}{0.8,0,0}

\newcommand{\nlc}{\normalcolor}

\begin{document}
\title{\textbf{ Anisotropic $p$-Laplacian Evolution \\ of Fast Diffusion type}\\[7mm]}

\author{\Large  Filomena Feo\footnote{Dipartimento di Ingegneria, Universit\`{a} degli Studi di Napoli
\textquotedblleft Parthenope\textquotedblright, Centro Direzionale Isola C4
80143 Napoli, \    E-mail: {\tt filomena.feo@uniparthenope.it}}  \,, \quad
\Large Juan Luis V\'azquez \footnote{Departamento de Matem\'aticas, Universidad Aut\'onoma de Madrid, 28049 Madrid, Spain. \newline
E-mail: {\tt juanluis.vazquez@uam.es}} \,,
\\[8pt] \Large  and
 \ Bruno Volzone \footnote{Dipartimento di Scienze e Tecnologie, Universit\`{a} degli Studi di Napoli ''Parthenope", Centro Direzionale Isola C4, 80143 Napoli, Italy. \    E-mail: {\tt bruno.volzone@uniparthenope.it}}}

\date{\ } 

\maketitle

\begin{abstract}
We study an anisotropic, possibly non-homogeneous version of the evolution $p$-Laplacian equation when fast diffusion holds in all directions.  We develop the basic theory and prove symmetrization results from which we derive $L^1$ to $L^\infty$  estimates. We prove the existence of a self-similar fundamental solution of this  equation in the appropriate exponent range, and uniqueness in a smaller range. We also obtain the asymptotic behaviour of finite mass solutions in terms of the self-similar solution. Positivity, decay rates as well as other  properties of the solutions are derived.  The combination of self-similarity and anisotropy is not common in the related literature. It is however essential in our analysis and creates mathematical difficulties that are solved for fast diffusions.

\end{abstract}

\medskip

\setcounter{page}{1}

\

\numberwithin{equation}{section}

\newpage

\section{Introduction}
This paper focusses on the study of the existence of  self-similar fundamental solutions  to the following ``anisotropic  $p$-Laplacian equation''(APLE for short),
 \begin{equation}\label{APL}
 u_t=\sum_{i=1}^N(|u_{x_i}|^{p_i-2}u_{x_i})_{x_i}\quad \mbox{posed in
 } \ \quad Q:=\mathbb{R}^N\times(0,+\infty),
 \end{equation}
 and their role to describe the long time behaviour of general classes of
finite-mass of the initial-value problem. Fundamental solutions are solutions of the equation  for all times $t>0$ that take a point mass (\emph{i.\,e.,} a Dirac delta) as initial data.
 In the process, we construct a theory of existence and uniqueness for initial data in $L^q$ spaces, $1\le q<+\infty$, and prove important results
on symmetrization, boundedness, barriers and positivity.

 We are specially interested in the presence of different growth exponents $p_i$. We take $N\geq2$ and $p_i>1$  for $i=1,...,N$. Therefore, this
equation is an anisotropic relative of the standard isotropic $p$-Laplacian equation
\begin{equation}\label{PLE}
u_t=\Delta_p u:=\sum_{i=1}^N (|\nabla u|^{p-2}u_{x_i})_{x_i},
\end{equation}
that has been extensively studied in the literature as the standard model
for gradient dependent nonlinear diffusion equation, with possibly degenerate or singular character.  Though most the attention has been given to the elliptic counterpart, $-\Delta_p u =f$, the parabolic case is also treated, see e.g. the well-known \cite{DiB93, Li69, Lind19} among the many references.

Even in the case where all the exponents $p_i$ in \eqref{APL} are the same, we obtain an alternative  version $u_t= L_{p,h} (u)$ with a homogeneous but  non-isotropic spatial operator,
 \begin{equation}\label{APLiso}
 L_{p,h} (u):=\sum_{i=1}^N (|u_{x_i}|^{p-2}u_{x_i})_{x_i},
 \end{equation}
which appears quite early in the literature, cf. \cite{Li69, Vi62, Vi63},
see also \cite{BGM96}.
This operator has been sometimes named ``pseudo $p$-Laplacian operator'',
\cite{BellKa2004}, and more recently , ``orthotropic $p$-Laplacian operator'', see \cite{BouBra,BBLV}, due to the invariance of $L_{p,h}$ with respect to the dihedral group for $N=2$. This will be our preferred denomination. The parabolic version appears in \cite{Rav67, Rav70}. In the general studies of nonlinear diffusion the case where the  exponents $p_i$ are different falls into the category of ``structure conditions with non-standard growth''.  The anisotropic equation was also studied in a number of references like \cite{LiZha01, Sango03}. Actually, a more general doubly nonlinear model was introduced in those references, see also \cite{AntCh08}. Very general structure conditions are considered by various authors
like \cite{Tolks84}, specially in elliptic problems. Our interest here differs from those works.


\medskip

\noindent {\bf The setting}. We consider solutions to the Cauchy problem for equation \eqref{APL} with nonnegative initial data
\begin{equation}\label{IC}
u(x,0)=u_0(x), \quad x\in \R^{N}.
\end{equation}
We assume that  $u_0\in L^1(\mathbb{R}^N)$, $u_0\ge 0$, and put $M:=\int_{\mathbb{R}^N} u_0(x)\,dx$, the so-called total mass.
The reader is here reminded that the strong qualitative and quantitative separation
between the two exponent ranges, $p>2$ and $p<2$ is a key feature of the isotropic $p$-Laplacian equation \eqref{PLE}. We recall that in the isotropic equation the range $p>2$ is called the slow gradient-diffusion case (with finite speed
of propagation and free boundaries), while the range $1<p<2$ is called the fast
gradient-diffusion case (with infinite speed of propagation), cf. \cite{DiB93}, and Section 11 of \cite{VazSmooth}.

\noindent $\bullet$ In this paper we will focus on the case where \sl fast diffusion holds in all directions\rm, i.e.,
\begin{equation}\label{H1}
1<p_i<2 \qquad \mbox{for all } \ i = 1,\cdots,N.  \tag{H1}
\end{equation}
We recall that in the orthotropic  fast diffusion equation (i.e., equation \eqref{APL} with $p_1 = p_2 =\cdots=p_N = p < 2$, hence $p$-homogeneous), there is a critical exponent,
\begin{equation}\label{p crit}
p_c(N):=\frac{2N}{N+1}
\end{equation}
such that $p>p_c$ is a necessary and sufficient condition for the existence of fundamental
solutions, cf. \cite{VazSmooth}. Note that $1<p_c(N)<2$ for $N\ge 2$.

\noindent $\bullet$ Moreover, we will always assume the  condition
\begin{equation}\label{p_isum}
\sum_{i=1}^{N}\frac{1}{p_i}<\frac{N+1}{2},\tag{H2}
\end{equation}
that is crucial in what follows. We we may also write it in terms of $p_c$  as: $\bar p>p_c$, where $\bar p$ is the inverse-average
\begin{equation}\label{p bar}
\frac{1}{\overline{p}}=\frac{1}{N}\sum_{i=1}^{N}\frac{1}{p_i}.
\end{equation}

We point out that (H2)  excludes the presence of (many) small exponents $1<p_i< p_c$ close to 1.
On the contrary, condition (H2) would  obviously be in force under the assumptions of slow diffusion in all directions:  $2\leq p_i<+\infty$ for all $i = 1,\cdots,N$ ( a situation we will not consider here).  However,
 in the fast diffusion  range  we have to impose it, otherwise the results we expect to obtain would be false.

\noindent $\bullet$ Finally, it is well known  in the literature on operators with non-standard growth that some control on the difference of diffusivity exponents is needed, see for instance \cite{BaCoMin, BDM, Marc20}. Here, we will only need the condition
\begin{equation}\label{H3}
p_i\leq \frac{N+1}{N}\bar p \tag{H3}
\end{equation}
(see Section \ref{sec.sss}). It is remarkable that  this  condition is automatically satisfied  if (H1) and (H2) are in force.

Under these conditions on the exponents, we develop a theory of existence, regularity, symmetrization, and upper and lower estimates for the Cauchy problem. We prove the existence of a self-similar solution starting from a Dirac mass, so-called fundamental solution or Barenblatt solution. Moreover, in  the particular orthotropic case $p_{i}=p$ for all $i$, thanks to extra regularity results that we derive, it is possible to prove uniqueness of the fundamental solution, and the theory goes on to show the asymptotic behaviour of all nonnegative finite mass solutions in the sense  that they are attracted by the corresponding Barenblatt solution with same mass as $t\to\infty$. This set of results shows that the ideas proposed by Barenblatt in his classical work \cite{Barbk96}  are valid for our
equation too.

\medskip

\noindent {\bf Outline of the paper by sections.} Here is a detailed summary of the contents. In Section \ref{sec.sss} we examine the form of the possible self-similar solutions, the a priori conditions on the exponents, and we also introduce the renormalized equation and its elliptic counterpart. The role of assumptions (H1), (H2) and (H3) is examined.

In Section \ref{sec.basic} we review the basic existence and uniqueness theory for the Cauchy Problem using the theories of monotone and accretive
operators in $L^q$ spaces. This general theory is valid in the whole range $p_i>1$, with no further restriction on the exponents.

In Section \ref{sec.symm} we develop the technique of Schwarz symmetrization for our anisotropic equation and we prove sharp comparison results by using the concept of mass concentration. This is an important topic in itself with a huge literature, specially when anisotropy is mild, see \cite{Bandle, Bear2019, Tal76}. The passage from anisotropic to isotropic is based on a sharp elliptic result by Cianchi \cite{Cianchi} that we develop in this setting using mass comparison, a strong tool used in some of our previous papers.  The topic of symmetrization has independent interest, and the theory and results are proved for all $p_i>1$ under assumption (H2).

The theory developed up to this point (including symmetrization) is used in Section \ref{sec.bdd} to  obtain a uniform $L^\infty$ bound for solutions with $L^1$ data, the so-called $L^1$-$L^\infty$ effect. Theorem \ref{L1LI}  is a key estimate in what follows.

We begin at this moment the construction of the self-similar fundamental solution under conditions (H1) and (H2). In a preparatory Section \ref{sec.upp} we construct the sharp anisotropic upper barrier for the solutions
of our problem, yet another key tool we need. The theory is now ready to tackle the construction of the special solution. The existence result, Theorem    \ref{thm.fundamental solution}, is maybe the main result of the paper. In Section \ref{sec.lbarr} we construct the lower barrier  and prove global positivity, an important additional information on the obtained
solution.

The very delicate question of uniqueness of the fundamental solutions is solved only for the orthotropic case, $p_i=p$, in Section \ref{sec.uniq}, and as consequence we establish the asymptotic behaviour of general solutions of the Cauchy problem in that case, see Section \ref{sec.asymp}. Both questions remain open for the anisotropic non-orthotropic equations.

As supplementary information, we discuss in Section \ref{sec.aniso}  the necessary control on the anisotropy for the theory to work. We devote Section \ref{sec.dne}  to introduce the study of self-similarity for Anisotropic Doubly Nonlinear Equations. Finally, we add a section on comments and open problems.

\medskip

\noindent {\bf Some related works}. This work follows the study of self-similarity for the anisotropic Porous Medium Equation (APME) in the fast diffusion range done by the authors in \cite{FVV2020}, where previous references to the literature are mentioned. Though it is well-known that the PME and the PLE are closely related as models of nonlinear diffusion of degenerate type, see for instance \cite{VazCIME}, the theories differ in many important details, hence the interest on this investigation.

In a  recent paper, Ciani and Vespri \cite{CV2020} study the existence of
Barenblatt solutions for the same Anisotropic $p$-Laplace Equation \eqref{APL} posed also in the whole space, but they consider the slow diffusion
case in all directions, i.e., $p_ i>2$ for all $i$. They exploit the property of finite propagation that holds in that exponent range. Uniqueness and asymptotic behaviour are not discussed. Their paper and ours contain parallel, non-overlapping  information.
Let us point out that the existence of fundamental solutions for anisotropic elliptic equations is a different issue, it has been studied by several authors like \cite{CirVet2015}.



\section{Self-similar solutions}\label{sec.sss}

We start our study by taking a closer look at the possible class of self-similar solutions.
This section follows closely the arguments of \cite{FVV2020} for the anisotropic Porous Medium Equation, but they lead to a quite different algebra, hence a careful analysis is needed. The common type of self-similar solutions of equation \eqref{APL} takes into account the anisotropy in the form
\begin{equation}\label{sss}
B(x,t)=t^{-\alpha}F(t^{-a_1}x_1,..,t^{-a_N}x_N),
\end{equation}
with constants $\alpha>0$, $a_1,..,a_n\ge 0$ to be chosen below by algebraic considerations. Indeed, if we substitute this formula into equation \eqref{APL} and write $y=(y_1,...,y_{N})$ and $y_i=x_i \,t^{-a_i}$,
equation \eqref{APL} becomes
$$-t^{-\alpha-1}\left[\alpha F(y)+
\sum_{i=1}^{N} a_iy_i\,F_{y_i}
\right]=\sum_{i=1}^{N}t^{-[\alpha(p_i-1)+p_ia_i]}\left(|F_{y_i}|^{p_i-2}F_{y_i}\right)_{y_i}.
$$
We see that time is eliminated as a factor in the resulting equation on the condition that:
\begin{equation}\label{ab}
\alpha(p_i-1)+p_ia_i=\alpha+1 \quad \mbox{ for all } i=1,2,\cdots,  N.
\end{equation}
We also look for integrable solutions that will enjoy the mass conservation property, and this
implies that \ $\alpha=\sum_{i=1}^{N}a_i$.
Imposing both conditions, and putting $a_i=\sigma_i \alpha,$ we get unique values for $\alpha$ and $\sigma_i$:
\begin{equation}\label{alfa}
\alpha=\frac{N}{N\overline{p}-2N+\overline{p}}
\end{equation}
and
\begin{equation}\label{ai} \ 
 \sigma_i= \frac{1}{p_i}\frac{(N+1)\bar p}{N}-1, \quad \mbox{\rm i.e., } \ \sigma_i-\frac1{N}= \frac{(N+1)}{N}\frac{(\overline{p}-p_i)}{p_i},
\end{equation}
so that \ $\sum_{i=1}^{N}\sigma_i=1. $ This is a delicate calculation
that produces the special value $\overline{p}$.

Observe that Condition (H2) is required to ensure that $\alpha>0$,  so that the self-similar solution will decay in time in maximum value like a power of time. This is a crucial condition for the self-similar solution to exist and play its role as asymptotic attractor, since the existence theory we present contains the maximum principle, hence the sup norm of
the constructed solutions cannot increase in time.

As for the $\sigma_i $ exponents that control the rate of spatial spread in each coordinate direction, we know that \ $\sum_{i=1}^{N}\sigma_i=1$, and in particular $\sigma_i=1/N$ in the homogeneous case. Condition
(H3) on the $p_i$  ensures that $\sigma_i> 0$. This means that the self-similar solution expands  as time passes (or at least, it does not contract), along any of the coordinate directions.

To fix ideas, we present in Section \ref{sec.aniso} a graphic analysis of
assumptions (H1), (H2), (H3) for general exponents $p_i>1$ in dimension $N=2$. We also compare this analysis with the predictions made in \cite{FVV2020} for the APME. \nc

\medskip

\noindent $\bullet$ With these choices,  the {\sl profile function} $F(y)$ must satisfy the following nonlinear anisotropic stationary equation in
$\mathbb{R}^N$:
\begin{equation}\label{StatEq}
\sum_{i=1}^{N}\left[\left(|F_{y_i}|^{p_i-2}F_{y_i}\right)_{y_i}+\alpha\sigma_i\left( y_i F\right)_{y_i}\right]=0.
\end{equation}
Conservation of mass must also hold : $\int B(x,t)\, dx =\int F(y)\, dy=M<\infty $ \ {\rm for all} \ $t>0.$
It is our purpose to prove that there exists a suitable solution of this elliptic equation, which is the anisotropic version of the equation of the Barenblatt profiles in the standard $p$-Laplacian, cf. \cite{VazSmooth}.

\medskip

 \noindent {\bf Examples.} {\bf 1) The isotropic case.} It is well-known that the source-type self-similar solution is indeed explicit in the isotropic case $u_t=\sum_{i=1}^N (| \nabla u|^{p-2}u_{x_i})_{x_i}$. Of course, for $p=2$ we obtain the Gaussian kernel of the heat equation: $ F(y)=(4\pi)^{-N/2}e^{-\frac{|y|^2}{4}}$. In the nonlinear cases we get
 two different but related formulas.

  $\bullet$ For $p_c<p<2$
 $$
 F(y)=\left( C_0+\frac{2-p}{p} \lambda^{-\frac{1}{p-1}}\,|y|^{\frac{p}{p-1}}\right)^{-\frac{p-1}{2-p}},
 $$
\ $\bullet$  when $p>2$ we get
 $$
 F(y)=\left( C_0- \frac{p-2}{p} \lambda^{-\frac{1}{p-1}} |y|^{\frac{p}{p-1}}\right)_+^{\frac{p-1}{p-2}},
 $$
with $\lambda=N(p-2)+p$ and  $C_0>0$ is an arbitrary  constant such that can be determined in terms of  the initial mass $M$.  They are called the Barenblatt solutions \cite{Barbk82}.

For $1<p\le 2$ the profile $F$ is everywhere positive, moreover for $p_c<p<2$ the profile $F$ belongs to $L^1(\mathbb{R}^N)$ and has a decay with a characteristic power rate. On the contrary, for $p>2$ the profile $F$ has compact support and exhibits a free boundary. Free boundaries are important objects for slow diffusion but they will appear in this paper only in passing.

\medskip

 {\bf 2) The orthotropic case.}  We have found a rather similar explicit formula for $F$ when $p_i=p$ for all $i$, so that $\overline{p}=p$. In that case we have

 $\bullet$ if $p_c<p<2$
\begin{equation}\label{Fp<2}
 F(y)=\left( C_0+ \frac{2-p}{p}\lambda^{-1/(p-1)}\sum_{i=1}^N\,|y_i|^{\frac{p}{p-1}}\right)^{-\frac{p-1}{2-p}} 
 ,
\end{equation}
with $C_0>0$ and $\lambda=N(p-2)+p$ as above. It is a solution to \eqref{StatEq}, because it solves
$$
|F_{y_i}|^{p-2}F_{y_i}+\frac{\alpha}{N} y_i F=0\quad \text{ in } \mathbb{R}^N\quad \text{ for all }i.
$$
Moreover, condition $p_c<p$ guaranties that $F\in L^1(\mathbb{R}^N)$.  Note that the constant $C_0>0$ is arbitrary and allows fixing the mass $M>0$ at will.

\medskip

$\bullet$ As a complement we state the case $p>2$
\begin{equation}\label{Fp>2}
 F(y)=\left( C_0-\frac{p-2}{p}\lambda^{-1/(p-1)}\sum_{i=1}^N\,|y_i|^{\frac{p}{p-1}}\right)_+^{ \frac{p-1}{p-2}}
 ,
\end{equation}
with $C_0>0$ and same $\lambda$. To our best knowledge, the explicit formulas \eqref{Fp<2} and \eqref{Fp>2} are new, as well as the formulas for $V$ below.

In order to fix the mass of $F$ given by \eqref{Fp<2} or \eqref{Fp>2} we
use transformation ${\mathcal T}_k [F(y)]=k F(k^{\frac{2-p}{p}}y)$ that
changes solutions into new solutions of the stationary equation \eqref{StatEq} with $p_i=p$ and changes the mass according to the rule \ $\int {\mathcal T}_k [F(y)]dy=k^{N+1-\frac{2N}{p}} \,\int F(z)\,dz$.

\medskip

 {\bf 3)} Putting $C_0=0$ in \eqref{Fp<2} we get for $p_c<p<2$
 the following parabolic solution
 $$
 V(x,t)=k_1\,t^{\frac1{2-p}}\left(\sum_{i=1}^N\,|x_i|^{\frac{p}{p-1}}\right)^{-\frac{p-1}{2-p}} \quad
 \text{for suitable $k_1>0$}.
 $$
 This is called a {\sl very singular solution}, since it contains a singularity with infinite integral at $x=0$. A much more singular solution  can be obtained by separating the variables
 $$
 V(x,t)=k_2\,t^{\frac1{2-p}}\,\left(\sum_{i=1}^N\,|x_i|^{{\color{magenta}-}\frac{p}{2-p}}\right) \quad
 \text{for suitable $k_2>0$}.
 $$

 \medskip

 {\bf 4)} We will not get any explicit formula for $F$ in the general anisotropic case, but we will have existence of self-similar solutions and suitable estimates, in particular decay.

 \nc
\subsection{Self-similar variables}
In several instances in the sequel, it will be convenient to pass to self-similar variables,
by zooming the original solution according to the self-similar exponents \eqref{alfa}-\eqref{ai}. More precisely, the change is done via the formulas
 \begin{equation}\label{NewVariables}
v(y,\tau)=(t+t_0)^\alpha u(x,t),\quad \tau=\log (t+t_0),\quad y_i=x_i(t+t_0)^{-\sigma_i\alpha} \quad i=1,..,N,
\end{equation}
with $\alpha$ and $\sigma_i$ as before.
We recall that all of these exponents are positive. There is a free time parameter $t_0\ge 0$ (a time shift).

\begin{lemma}\label{Lem1}
If $u(x,t)$ is a solution (resp. super-solution, sub-solution) of \eqref{APL}, then $v(y,\tau)$ is a solution (resp. super-solution, sub-solution)
of
\begin{equation}\label{APLs}
v_\tau=\sum_{i=1}^N\left[(|v_{y_i}|^{p_i-2}v_{y_i})_{y_i}+\alpha \sigma_i \left(\,y_i\,v\right)_{y_i}\right] \quad \mathbb{R}^N\times(\tau_0,+\infty).
\end{equation}
\end{lemma}
This equation will be a key tool in our study. Note that the rescaled equation does not change with the time-shift  $t_0$ \nlc but the initial value in the new time does, $\tau_0 = \log(t_0)$. Thus, if $t_0 = 1$
then $\tau_0=0$.  If $t_0 = 0$ then $\tau_0=-\infty$ and the $v$ equation is defined for $\tau\in \mathbb{R}$.

We stress that this change of variables preserves the $L^1$ norm. The mass of the $v$ solution at new time $\tau\geq\tau_0$ equals that of the $u$
at the corresponding time $t\geq0$.

\medskip

$\bullet$ This equation enjoys a scaling transformation $\mathcal T_k$ that changes the mass:
\begin{equation}\label{mass.trans}
 \mathcal T_k[v(y,\tau)]= k\,v(k^{\beta_1}y_1, \cdots, k^{\beta_N}y_N, \tau) \quad\quad \beta_i=\frac{2-p_i}{p_i}
\end{equation}
with scaling parameter $k>0$. Working out the new mass we get
$$
\int_{\mathbb{R}^N} \mathcal T_k[v(y,\tau)]\, dy=\int_{\mathbb{R}^N}  v(y,\tau)\, dy
$$
with $ \mu=1-\sum_i \beta_i=N+1-\sum_i (2/p_i)=(N+1)-(2N/\overline p)$. We have $\mu>0$ since $\overline p>p_c$.


\section{Basic theory. Variational setting}\label{sec.basic}

The theory of the anisotropic $p$-Laplacian operator \eqref{APL} shares a
number of basic features with its best known relative, the standard isotropic $p$-Laplacian $\Delta_p$. These common traits have been already mentioned in the literature in the  case of anisotropy with same powers, but we will see here that the similarities extend to the general form. The only assumption we make in this setting is that \ $p_{i}>1$ for all $i=1,...,N$. We denote by $X^{\overrightarrow{p}}$ the anisotropic Banach space
\begin{equation}\label{X}
X^{\overrightarrow{p}}=\left\{u\in L^{2}(\R^{N}):\,u_{x_{i}}\in L^{p_{i}}(\R^{N}),\,\forall i=1,...,N\right\}
\end{equation}
endowed with the norm
\[
\|u\|_{X^{\overrightarrow{p}}}=\|u\|_{L^{2}}+\sum_{i=1}^{N}\|u_{x_{i}}\|_{L^{p_{i}}}.
\]
\medskip
It is easy to see that $C_{0}^{\infty}(\R^{N})$ is dense in $X^{\overrightarrow{p}}$ and that $X^{\overrightarrow{p}}$ reduces to $H^{1}(\R^{N})$ when $p=2$.

\indent $\bullet $ Let us consider the anisotropic operator
\begin{equation}
\mathcal{A}(u):=-\sum_{i=1}^N (|u_{x_i}|^{p_i-2}u_{x_i})_{x_i}\,,\label{A}
\end{equation}
defined on the domain
\[
D(\mathcal{A})=\left\{u\in X^{\overrightarrow{p}}:\,\mathcal{A}(u)\in L^{2}(\R^{N})\right\}.
\]
It is easy to see that $\mathcal{A}:D(\mathcal{A})\subset L^{2}(\R^{N})\rightarrow L^{2}(\R^{N})$ is the subdifferential of the convex functional
\begin{equation}\label{Jfunc}
\mathcal J (u)=
\ \left\{
\begin{array}
[c]{ll}%
\sum_{i=1}^N \frac1{p_i}\int_{\R^{N}} |u_{x_i}(x)|^{p_i}\,dx,  & \mbox{
 if }u\in X^{\overrightarrow{p}}\\[6pt]
& \\
+\infty & \mbox{ if }u\in L^{2}(\R^{N})\setminus X^{\overrightarrow{p}},
\end{array}
\right.
\end{equation}
whenever $p_i>1$ for all $i$. Then we have that the domain of $\mathcal J$ is $D( \mathcal {J} )=X^{\overrightarrow{p}}$. Now we use the theory of maximal monotone operators of \cite{BrezisOMM} (see also the monograph \cite{BarbuBk} and Chapter 10 of
\cite{Vlibro} for a summary and its application to the Porous Medium Equation). Let us prove some important facts, which follow from classical variational arguments.
Thus, we  can  solve the nonlinear elliptic equation
\begin{equation}\label{ell.eq.L}
 \lambda\mathcal{A}u+ u=f
 \end{equation}
in a unique way for all $f\in L^2(\ren)$ and all $\lambda>0$, with solutions $u\in D(\mathcal{A})$.

\begin{proposition}\label{existuniqstatprob}
For all $\lambda>0$ and $f\in L^{2}(\R^{N})$ there exists a unique strong
solution
$u\in X^{\overrightarrow{p}}$ of \eqref{ell.eq.L}. Moreover, the $T$-contractivity holds: if $f_{1},\,f_{2}\in L^{2}(\R^{N})$ and $u_{1},\,u_{2}$ solve \eqref{ell.eq.L} with datum $f_{1}$, $f_{2}$ respectively, we have
\begin{equation}\label{L2_Tcontr}
\int_{\R^{N}} (u_1-u_2)_+^2\,dx\le \int_{\R^{N}} (f_{1}-f_{2})^2_+\,dx\,,
\end{equation}
where $(f)_+=\max\{f(x),0\}$. Finally, a comparison  principle applies in the sense that $f_1\ge f_2$ {\textit{}{a.e.} in $\R^{N}$} implies $u_1\ge u_2$ {\textit{}{a.e.} in $\R^{N}$}.
 \end{proposition}

\begin{proof}
Let us define the functional
\[
\mathsf{J}(u)=\lambda\sum_{i=1}^{N}\frac{1}{p_{i}}\int_{\R^{N}}|u_{x_{i}}|^{p_{i}}dx+\frac{1}{2}\int_{\R^{N}}u^{2}dx
-\int_{\R^{N}}fu\,dx
\]
for any $u\in X^{\overrightarrow{p}}$. It is clear that $\mathsf{J}$ is strictly convex, thus if a minimizer exists, it is the unique weak solution to \eqref{ell.eq.L}. Let us prove that $\mathsf{J}$ is bounded from below. For any $u\in X^{\overrightarrow{p}}$ we have, by Young's inequality,
\[
\mathsf{J}(u)\geq \lambda\sum_{i=1}^{N}\frac{1}{p_{i}}\int_{\R^{N}}|u_{x_{i}}|^{p_{i}}dx+\left(\frac{1}{2}-\varepsilon\right)\int_{\R^{N}}u^{2}dx-C(\varepsilon)\int_{\R^{N}}f^{2}\,dx
\]
hence choosing $\varepsilon<1/2$
\[
\mathsf{J}(u)\geq -C(\varepsilon)\int_{\R^{N}}f^{2}\,dx.
\]
Now if $\left\{u_{n}\right\}\subset X^{\overrightarrow{p}}$ is a minimizing sequence of $\mathsf{J}$ it easily follows that
\[
\|u_{n}\|^{2}_{L^{2}(\R^{N})}\leq2\mathsf{J}(u_{n})+2\int_{\R^{N}}fu_{n}\,dx,
\]
then Young's inequality again provides
\[
(1-2\varepsilon)\|u_{n}\|^{2}_{L^{2}(\R^{N})}\leq 2\mathsf{J}(u_{n})+C(\varepsilon)\int_{\R^{N}}f^{2}dx,
\]
then by uniform boundedness of $\mathsf{J}(u_{n})$ the sequence $\left\{u_{n}\right\}\subset X^{\overrightarrow{p}}$  is bounded in $L^{2}(\R^{N})$, thus it admits a subsequence, which we still relabel $\left\{u_{n}\right\}$, weakly converging to some $u\in L^{2}(\R^{N})$. Now we observe that
\[
\lambda\frac{1}{p_{i}}\int_{\R^{N}}|\partial_{x_i} u_{n}|^{p_{i}}dx\leq \mathsf J(u_{n})+\int_{\R^{N}}f u_{n}\,dx \quad \text{ for every } i=1,...,N
\]
and since $\mathsf{J}(u_{n})$ is uniform bounded and $\left\{u_{n}\right\}$ is bounded in $L^{2}(\R^{N})$ we have that $\left\{\partial_{x_i}u_{n}\right\}$ is bounded in $L^{p_{i}}(\R^{N})$ for all $i=1,...,N$. Thus up to subsequences it follows $\partial _{x_{i}}u_{n}\rightharpoonup g_{i}$ weakly in $L^{p_{i}}(\R^{N})$, for each $i=1,...,N$. Since $u_{n}$ converges weakly in $L^{2}(\R^{N})$ to $u$ we find $g_{i}=\partial_{x_{i}}u$ for all $i=1,...,N$. By the lower semi-continuity of the $L^{q}(\R^{N})$ norms we then obtain
\begin{align*}
\liminf_{n\rightarrow\infty}\mathsf{J}(u_{n})&=\liminf_{n\rightarrow\infty}\left( \lambda\sum_{i=1}^N \frac1{p_i}\int_{\R^{N}} |\partial_{x_{i}}u_{n}|^{p_i}\,dx+\frac{1}{2}\int_{\R^{N}}u_{n}^{2}dx-\int_{\R^{N}}fu_{n}\,dx\right)
\\
&\geq \lambda\sum_{i=1}^{N}\frac{1}{p_{i}}\liminf_{n\rightarrow\infty}\int_{\R^{N}} |\partial_{x_{i}}u_{n}|^{p_i}\,dx +\frac{1}{2}\liminf_{n\rightarrow\infty}\int_{\R^{N}}u_{n}^{2}dx-\int_{\R^{N}}fu\,dx
\\
& \geq \lambda\sum_{i=1}^{N}\frac{1}{p_{i}}\int_{\R^{N}} |\partial_{x_{i}}u|^{p_i}\,dx+\frac{1}{2}\int_{\R^{N}}u^{2}dx-\int_{\R^{N}}f u\,dx=\mathsf{J}(u),
\end{align*}
 therefore $u$ is the unique minimizer of $\mathsf{J}$. In order to prove
the $T$ contraction, as usual we multiply by $(u_{1}-u_{2})_{+}$ the difference of the equations related to data $f_{1}$ and $f_{2}$ and integrate
in space. We are able to conclude using monotonicity of $\mathcal{A}$.

 Note that $\mathcal{A}(u)=(f-u)/ \lambda$, so we have $u\in  D(\mathcal A)$. The solution is therefore a strong solution.
\end{proof}

\begin{remark}
Proposition \ref{existuniqstatprob} holds if $f$ belongs to  the dual space of $X^{\overrightarrow{p}}$, where the dual norm replaces the $L^2$ norm at the right-hand side of \eqref{L2_Tcontr}.
\end{remark}

Note: this also applies for the problem posed in a bounded domain $\Omega$, and then the natural boundary condition is $u(x)\rightarrow0\quad\text{as }|x|\to\partial \Omega.$


By Proposition \ref{existuniqstatprob} we have that $R(I+\lambda\mathcal{A})=L^{2}(\R^{N})$ and the resolvent operator $R_{\lambda}(\mathcal{A})=(I+\lambda \mathcal{A})^{-1}:L^{2}(\R^{N})\rightarrow D(\mathcal{A})$ is onto and a contraction for all $\lambda>0$. Hence \cite[Proposition 2.2]{BrezisOMM} implies that $\mathcal{A}$ is a maximal monotone operator in $L^2(\R^{N})$ (in other words, $\mathcal{A}$ is maximal dissipative).
\medskip

\indent $\bullet $ Recall that $\mathcal{A}$ is the subdifferential of the convex functional
$\mathcal J (u)$, where $\mathcal{J}$ is lower semi-continuous on $L^{2}(\R^{N})$ (indeed it can be easily proven that its sublevel sets are strongly closed in $L^{2}(\R^{N})$, following some arguments of Proposition \ref{existuniqstatprob}). Hence, it follows from \cite[Theorem 3.1, Theorem
3.2]{BrezisOMM} that we can solve the evolution equation
\begin{equation}\label{parb.eq}
u_t=-\mathcal{A}(u)
 \end{equation}
for all initial data $u_0\in L^{2}(\R^{N})$. We observe that $D(\mathcal{A})$ is dense in $L^{2}(\R^{N})$, in other words we can construct the gradient flow in all of $L^{2}(\R^{N})$ \nc corresponding to the functional $\mathcal{J}$. In particular, the solution $u:[0,+\infty)\rightarrow L^{2}(\R^{N})$ is such that $u(t)\in D(\mathcal{A})$ for all $t>0$, this map
is Lipschitz in time, it solves equation \eqref{parb.eq} pointwise on $\R^{N}$ for a.e. $t>0$ and $u(0)=u_{0}$. Moreover the semigroup maps $S_t^\mathcal{A} :u_0\mapsto u(t)$ form a continuous semigroup of contractions in $L^2(\ren)$. Comparison  principle and $T$-contractivity hold in the
sense that
\begin{equation}\label{L2_Tcontr.par}
\int_{\R^{N}} (u_1(t)-u_2(t))_+^2\,dx\le \int_{\R^{N}} (u_{0,1}-u_{0,2})^2_+\,dx\,.
\end{equation}
We call $S_t^\mathcal{A}$ the semigroup generated by $\mathcal{J}$ and the corresponding function $u(\cdot,t)=S_t^\mathcal{A}(u_0)$ is called the semigroup solution of the evolution problem (or more precisely the  $L^2$ semigroup solution). In particular, $u$ solves the partial differential equation \eqref{parb.eq} in the sense of {\sl strong solutions} in $L^2(\ren)$, \emph{i.e.} it agrees with the following definition:

\noindent {\bf Definition.} If $X$ is a Banach space, a function $u\in C((0,T);X)$ is called a strong solution of the abstract ODE: $u_t=-\mathcal{A}u$ if it is absolutely differentiable as an $X$-valued function of time for a.e. $t>0$, and moreover $u(t)\in D(L)$ and $du/dt=Lu$ for almost all times.
The theory says  that when $X$ is a Hilbert space and $\mathcal{A}$ is a subdifferential then the semigroup solution is a strong solution and $u(t)\in D(A)$ for all $t>0$. When $u_{0}\in L^{2}(\R^{N})$, since $D(\mathcal{A})$ is dense $L^{2}(\R^{N})$, we can use this theory to get strong solutions for every initial datum in that class.

\medskip

\indent $\bullet $ The semigroup solution has extra regularity in anisotropic Sobolev spaces by virtue of the following two computations, see \cite[Theorem 3.2]{BrezisOMM} :
\begin{equation}\label{firstenergy}
\frac12\frac{d}{dt}\|u(t)\|_2^2= -\langle  \mathcal{A}u(t),  u(t)\rangle_{L^{2}}=-\sum_{i=1}^N \int_{\ren} |u_{x_i}(x)|^{p_i}\,dx \leq -(\min_i{p_i})\,\mathcal{J}(u(t)).
\end{equation}

Moreover we have the following entropy-entropy dissipation identity
\begin{equation}\label{secondenergy}
\frac{d}{dt}{\mathcal J}(u(t)) =\langle  \mathcal{A}u(t), u_t(t)\rangle=-\|u_t(t)\|^2_2,
\end{equation}
where the norms are taken in $\ren$. It follows that both $\|u(t)\|_2$ and ${\mathcal J}(u(t))$ are decreasing in time. Then from \eqref{firstenergy} integrating on $(0,t)$ we get the estimate
\begin{equation}\label{esttimeder1}
{\mathcal J}(u(t))\le C\|u_0\|_2^2/t \quad \mbox{ for every } t>0,
\end{equation}
and from \eqref{secondenergy} integrating on $(t_1,t_2)$
\begin{equation}
\int_{t_1}^{t_2}\int_{\ren} u_t^2(x,t)\,dxdt \le {\mathcal J}(u(t_1)).\label{esttimeder2}
\end{equation}
This Sobolev regularity gives the compactness for times $t\ge \tau>0$ that we will need in Subsection \ref{sec.asymp}.

\medskip

\indent $\bullet $In this work we will also need an important extra  property of the $L^2$ semigroup which is the property of generating a contraction semigroup with respect to the norm of $L^q(\ren)$ for all $q\ge 1$, in particular for $q=1$. The $q$-semigroup in such a norm is defined first by restriction of the data to
$L^2(\ren)\cap L^q(\ren)$ and then is it extended to $L^q(\ren) $ by the technique of continuous extension of bounded operators. We leave the details to the reader, since it is well-known theory, but see next section.

\medskip

We will concentrate in the sequel on the semigroup solutions corresponding to data $u_0\in L^1(\ren)$, which we may call $L^1$ semigroup solutions. Apart from existence, uniqueness and comparison, we will need three extra properties: boundedness for positive times and comparison with super- and subsolutions defined in a  suitable way.

For future reference, let us state a general decay result.

\begin{proposition}\label{decay of the $L^p$}
If $u_0\in L^q(\mathbb{R}^N)$ for $q\in[1,+\infty]$, then the $L^q$ norms
$\|u(t)\|_{q}$ are non-increasing in time.
\end{proposition}

Two reminders about related results. First, the variational theory applies in bounded domains with suitable boundary data.

\begin{remark} The semigroup theory applies to Dirichlet boundary problem
defined in a bounded domain $\Omega$ as well with zero boundary data.
\end{remark}

We can also consider equations with a right-hand side.

\begin{remark} The complete evolution equation
$$
u_t+ \mathcal{A}(u)=f,
$$
including a forcing term can also be treated with the same maximal monotone theory when $f\in L^2(0,T:L^2(\ren)$ or $f\in L^2(0,T:L^2(\Omega)$ \end{remark}

We will not need such developments here. In the last case we do not get a
semigroup but a more complicated object $u=u(x,t; u_0,f)$.


\section{The $L^1$ theory}\label{existence}

In this section we will extend to the framework of the $L^1(\ren)$ space the existence result for solutions to the Cauchy problem for the full anisotropic  equation \eqref{APL}. This amounts in practice  to extending the contraction semigroup defined in $L^2(\ren)$ in the previous section to
a contraction semigroup in  $L^1(\ren)$, an issue that has been studied in some detail in the literature on linear and nonlinear semigroups, see \cite{CL86, CP72, Evans78, Pazy83, Show97}.  We will work for simplicity under the assumptions (H1)-(H2) (but see Remark \ref{RemarkIp}).

For the reader's benefit we will present the most important details. Experts may skip this section.
 The extension  will be done by means of nonlinear semigroup theory in Banach spaces and using the results of previous section in Hilbert spaces. We will provide the existence of a \emph{mild} solution by solving the \emph{implicit time discretization scheme} (ITDS for short).  Since the ITDS, as we see below, is based on the existence and uniqueness of solutions
to the stationary elliptic problem with a zero order term, we will first recollect briefly some information concerning the problem
\begin{equation}\label{Pf1 RN}
\left\{
\begin{array}
[c]{ll}%
-\sum_{i=1}^N (|u_{x_i}|^{p_i-2}u_{x_i})_{x_i}+ \mu u=f & \quad
\hbox{in $\mathbb{R}^N$}\\[6pt]
{u(x)\rightarrow0} & \quad \hbox{ as } |x|\rightarrow\infty
\end{array}
\right.
\end{equation}
for arbitrary constant $\mu>0$.
\begin{theorem}\label{ellipticexi}
Assume $f\in L^{1}(\R^{N})$ and $\mu>0$. Then there is a unique strong solution $u\in L^1(\mathbb{R}^N)$ to \eqref{Pf1 RN}. Moreover, the following $L^1$ contraction principle holds:
if $f_{1},\, f_2\in L^{1}(\R^{N})$ and $u_1$, $u_2$ are the corresponding
solutions, we have
\begin{equation}\label{elliptL1_contr}
\int_{\R^{N}} (u_1-u_2)_+\,dx\le \int_{\R^{N}} (f_{1}-f_{2})_+\,dx\,.
\end{equation}
In particular, if $f_{1}\leq f_{2}$ we have $u_1\le u_2$ a.e. .
\end{theorem}


\begin{proof}
We can proceed by approximation. Let us denote $T_k(s):=\min\{|s|,|m|\}\text{sign} s$ and  let us take $f_k=T_k(f)\in L^2(\mathbb{R}^N)\cap L^1(\mathbb{R}^N)$ such that $f_k\rightarrow f$ in $L^1(\mathbb{R}^N)$ and $\|f_k\|_{L^1(\mathbb{R}^N)}\leq\|f\|_{L^1(\mathbb{R}^N)}$ as a datum in \eqref{Pf1 RN} .

\medskip

i) Let $u^1_k$ and $u^2_k$ two solutions of the approximate problems with
respectively data $f^1_k$ and $f^2_k$ in $L^2(\mathbb{R}^N)$. Following  \cite[Prop. 9.1]{Vlibro}, let $p(s)$ a smooth approximation of the positive part of the sign function $\text{sign}(s)$, with $p(s)=0$ for $s\leq0$, $0\leq p(s)\leq 1$ for all $s\in\R$ and $p^{\prime}(s)\geq0$ for all $s\geq0$. Take any cutoff function $\zeta\in C_{0}^{\infty}(\R^{N})$, $0\leq\zeta\leq1$, $\zeta(x)=1$ for $|x|\leq1$, $\zeta(x)=0$ for $|x|\geq 2$ and set $\zeta_{n}(x)=\zeta(x/n)$ for $n\geq1$, so that $\zeta_{n}\uparrow 1$ as $n\rightarrow\infty$. Using $p(u^1_k-u^2_k)\,\zeta_{n}(x)$
as test function in the difference of equations and letting $p$ tend to $\text{sign}^{+}$
we get
\begin{equation*}
\begin{split}
\sum_{i=1}^N&\int_{\R^{N}}
\left(|\partial_{x_i}u^1_k|^{p_i-2}\partial_{x_i}u^1_k-|\partial_{x_i}u^2_k|^{p_i-2}\partial_{x_i}u^{2}_k\right)_{x_i\,}
\text{sign}^{+}(u^1_k-u^2_k)\,\zeta_{n}(x)dx
+
\\&\mu\int_{\R^{N}}
\left(u^1_k-u^{2}_k\right)
\text{sign}^{+}(u^1_k-u^2_k)\,\zeta_{n}(x)dx
=\int_{\R^{N}}
\left(f^1_k-f^2_k\right)
\text{sign}^{+}(u^1_k-u^2_k)\,\zeta_{n}(x)dx.
\end{split}
\end{equation*}
Now the monotonicity of operator gives
\begin{equation*}
\begin{split}
\mu &\int_{\R^{N}}
\left(u^1_k-u^{2}_k\right)
\text{sign}^{+}(u^1_k-u^2_k)\,\zeta_{n}(x)dx
\leq \int_{\R^{N}}
\left(f^1_k-f^2_k\right)_+\,\zeta_{n}(x)dx
\\
&- \sum_{i=1}^N\int_{\R^{N}}
\left(|\partial_{x_i}u^1_k|^{p_i-2}\partial_{x_i}u^1_k-|\partial_{x_i}u^2_k|^{p_i-2}\partial_{x_i}u^2_k\right)
\text{sign}^{+}(u^1_k-u^2_k)\,\partial_{x_i}\zeta_{n}(x)dx
\end{split}
\end{equation*}
We let now $n\rightarrow \infty$ to obtain
\begin{equation}\label{L1 contr approx}
\int_{\R^{N}} (u^1_k-u^2_k)_+\,dx\le \int_{\R^{N}} (f^{1}_k-f^{2}_k)_+\,dx\,,
\end{equation}
since the right-hand side goes to zero. Indeed we have
\begin{equation*}
\begin{split}
\sum_{i=1}^N\int_{\R^{N}}&
\left(|\partial_{x_i}u^1_k|^{p_i-2}\partial_{x_i}u^1_k-|\partial_{x_i}u^2_k|^{p_i-2}\partial_{x_i}u^2_k\right)
\text{sign}^{+}(u^1_k-u^2_k)\,\partial_{x_i}\zeta_{n}(x)dx
\\
&\leq
\sum_{i=1}^N
\left(
\int_{\R^{N}}
\left(|\partial_{x_i}u^1_k|^{p_i-2}\partial_{x_i}u^1_k-|\partial_{x_i}u^2_k|^{p_i-2}\partial_{x_i}
u^2_k\right)^{p_i'}
dx
\right)^{\frac{1}{p_i'}}
\frac{1}{n}\left(
\int_{\R^{N}}
\partial_{x_i}\zeta^{p_i}_{n}(x)dx
\right)^{\frac{1}{p_i}}
\\
&
\leq\sum_{i=1}^N
\left(
\int_{\R^{N}}
\left(|\partial_{x_i}u^1_k|^{p_i-2}\partial_{x_i}u^1_k-|\partial_{x_i}u^2_k|^{p_i-2}\partial_{x_i}
u^2_k\right)^{p_i'}
dx
\right)^{\frac{1}{p_i'}}
\frac{1}{n}\|\partial_{x_i}\zeta_{n}\|_{\infty}\left(
\int_{n<|x|<2n}dx
\right)^{\frac{1}{p_i}}
\end{split}
\end{equation*}
and that $\left(|\partial_{x_i}u^1_k|^{p_i-2}\partial_{x_i}u^1_k-|\partial_{x_i}u^2_k|^{p_i-2}\partial_{x_i}
u^2_k\right)^{p_i'}\in L^1(\R^{N})$.

\medskip

ii) By \eqref{L1 contr approx} it follows that $\{u^j_k\}$ is a Cauchy sequence in $L^1(\mathbb{R}^N)$, then $u^j_k\rightarrow u^j$ in $L^1(\mathbb{R}^N)$ for $j=1,2$ and we can pass to the limit in \eqref{L1 contr approx} obtaining \eqref{elliptL1_contr}.

\medskip

iii) Using $T_m(u_k)$ as test function in problem with datum $f_k$ we get
the following a priori estimate
\begin{equation*}\label{stima 2}
\sum_{i=1}^N\int_{\mathbb{R}^N} \left|\left(T_m(u_k)\right)_{x_i}\right|^{p_i}\, dx+\mu\int_{\mathbb{R}^N}\left(T_m(u_k)\right)^2 \, dx\leq m C(N,p_1,\cdots,p_N,\|f\|_{L^1(\mathbb{R}^N)})
\end{equation*}
for every $m>0$. 
By an anisotropic version of Lemma 4.1 and 4.2 of \cite{B6L1 theory} we have
\begin{equation}\label{stima 22}
\sum_{i=1}^N
\|\left(u_k\right)_{x_i}\|_{M^{s_i}(\mathbb{R}^N)}\leq C(N,p_1,\cdots,p_N,\mu,\|f\|_{L^1(\mathbb{R}^N)})
\end{equation}

where $M^{s_i}$ denote the Marcinkiewicz (or weak-$L^{s_i}$) spaces and $s_i=\frac{N'}{\bar p'}p_i$ 
for $i=1,\cdots,N$.

%

\medskip

When $s_i>1$ for all $i$, estimate \eqref{stima 22} yields that sequence $\{\partial_{x_i}u_k\}$ is bounded in $L^{q_i}_{loc}(\mathbb{R}^N)$ with $1< q_i<\frac{N'}{\bar p'}p_i$. Then (up a subsequence) $\partial_{x_i}u_k\rightarrow \partial_{x_i}u$ weakly in $L^{q_i}_{loc}(\mathbb{R}^N)$ and
$u\in L^1(\mathbb{R}^N)\cap W^{1,1}_{loc}(\mathbb{R}^N)$ is a distributional solution to \eqref{Pf1 RN}. Moreover we get $u_{x_i}\in M^{\frac{N'}{\bar p'}p_i}(\mathbb{R}^N)$ and $u\in M^{\frac{N(\bar p-1)}{N-\bar p}}(\mathbb{R}^N)$, because
\begin{equation}\label{stima 23}
\| u_k\|_{M^{\frac{N(\bar p-1)}{N-\bar p}}(\mathbb{R}^N)}\leq C(N,p_1,\cdots,p_N,\mu,\|f\|_{L^1(\mathbb{R}^N)}).
\end{equation}

 When at least one $s_i\leq 1$ and $\bar p>p_c$ we have to consider a different notion of solution, see e.g. \cite{B6L1 theory} for entropy solution's one. Following \cite{B6L1 theory} there exists a unique entropy solution  and $\frac{\partial}{\partial x_i}T_m(u)\in L^{p_i}(\mathbb{R}^N)$ and $u\in L^1(\mathbb{R}^N)\cap M^{ \frac{N(\bar p-1)}{N-\bar p}}(\mathbb{R}^N)$ by \eqref{stima 23}.
\end{proof}

In order to obtain the existence of solutions to the nonlinear parabolic problem we use the Crandall-Liggett theorem \cite{CL71} see also Chapter 10 of \cite[Chapter 10]{Vlibro}, which we briefly recall here in the abstract framework. Let $X$ be a Banach space and $\mathcal{A}:D(\mathcal{A})\subset X\rightarrow X$ a nonlinear operator defined on a suitable subset
of $X$. We start from the abstract Cauchy problem
\begin{equation}\label{eqcauchyabstract.3}
\left\{
\begin{array}
[c]{lll}%
u^{\prime}(t)+\mathcal{A}(u)=f,  &  & t>0,%
\\[4pt]
u(0)=u_{0}\,, &  &
\end{array}
\right.
\end{equation}
where $u_{0}\in X$ and $f\in L^{1}(0,T;X)$ for some $T>0$. We first take a partition of the interval, say, $t_k=kh$ for $k=0,1,\ldots n$ and $h=T/n$, and then we solve the ITDS, made by the system of difference relations
\begin{equation}
\frac{u_{h,k}-u_{h,k-1}}{h}+\mathcal{A}(u_{h,k})=f_{k}^{(h)}\label{discrellprob}
\end{equation}
for $k=0,1,\ldots n$, where we set $u_{h,0}=u_{0}$. The data set $\left\{f_{k}^{(h)}:k=1,\ldots,n\right\}$ is supposed to be a discretization
of the source term $f$, satisfying the relation
\[
\|f^{(h)}-f\|_{L^{1}(0,T;X)}\rightarrow0\quad\text{as }h\rightarrow0.
\]
The discretization scheme is then rephrased in the form
\[
u_{h,k}=J_{h}(u_{h,k-1}+hf_{k}^{(h)})
\]
where
\[
J_{\lambda}=(I+\lambda \mathcal{A})^{-1},\,\lambda>0
\]
is called the \emph{resolvent operator}, being $I$ the identity operator.
When the ITDS is solved, we construct a \emph{discrete approximate solution}
$\left\{u_{h,k}\right\}_{k}$, which is the piecewise constant function $u_h(t)$,
defined (for instance) by means of
\begin{equation}
u_{h}(t)=u_{h,k}\quad\text{if }t\in[(k-1)h,kh]\label{interpol}.
\end{equation}
If the operator $\mathcal{A}$ is $m$-accretive, we have that for all $u_{0}\in \overline{D(\mathcal{A})}$, the abstract  problem
\eqref{eqcauchyabstract.3} has a unique \emph{mild solution} $u$, \emph{i.e.} a function $u\in C([0,T);X)$ which is obtained as uniform limit of
approximate solutions of the type $u_{h}$, as $h\rightarrow0$, where the initial datum is taken in the sense that $u(t)$ is continuous in $t=0$ and $u(t)\rightarrow u_{0}$ as $t\rightarrow0$. We have then as $h\rightarrow0$:
 \[
 u(t):=\lim_{h\rightarrow0}u_{h}(t)\,,
 \]
 and the limit is always uniform in compact subsets of $[0,\infty)$.  Then we can prove the following parabolic existence-uniqueness result:
 \begin{theorem}\label{parabolicexistence}
 Let $0<T\leq + \infty$ and $Q_{T}:=\R^{N} \times(0,T)$. For any $u_{0}\in L^{1}(\R^{N})$ and any $f\in L^{1}(Q)$ there is a unique mild solution to the Cauchy problem
  \begin{equation}\label{Pf}
\left\{
\begin{array}
[c]{lll}%
u_t-\sum_{i=1}^N (|u_{x_i}|^{p_i-2}u_{x_i})_{x_i}=f  &  & \text{ in }
Q
\\[6pt]
u(x,0)=u_{0}(x) &  &   \text{ in } \R^{N}.
\end{array}
\right. %
\end{equation}
Moreover for every two solutions $u_1$ and $u_2$ to \eqref{APL} with respectively initial data $u_{0,1}$ and $u_{0,2}$ in $L^{1}(\R^{N})$ and source terms $f_{1},\,f_{2}\in L^{1}(Q_{T})$ we have for any $0\leq s\leq t<T$
\begin{equation}\label{L1_contr}
\int_{\R^{N}} (u_1(t)-u_2(t))_+\,dx\le \int_{\R^{N}} (u_{1}(s)-u_{2}(s))_+\,dx+\int_{s}^{t}\left[u_{1}(\tau)-u_{2}(\tau),f_{1}(\tau)-f_{2}(\tau)\right]_{+}\,d\tau,
\end{equation}
with the Sato bracket notation
\[
\left[v,w\right]_{+}=\inf_{\lambda>0}\frac{\|(v+\lambda w)_{+}\|_{L^{1}}-\|w_{+}\|_{L^{1}}}{\lambda}
\]
In particular, if $u_{0,1}\le u_{0,2}$ and $f_{1}\le f_{2}$ a.e., then for every $t>0$ we have $u_1(t)\le u_2(t)$ a.e..
\end{theorem}
\begin{proof}
 In order to apply the abstract theory recalled above, we introduce the nonlinear operator
$\mathcal{A}: D(\mathcal{A})\subset L^{1}(\ren)\rightarrow L^{1}(\ren)$, defined by \eqref{A}
with domain
$$
D(\mathcal{A} ):=\left\{v\in L^1(\ren):\,v_{x_{i}}\in M^{s_{i}}(\ren),\, \mathcal{A}(v)\in L^1(\ren) \right\},
$$
where we recall that $s_i=\frac{N'}{\bar p'}p_i$. By Theorem \ref{ellipticexi} we see that this operator is $T$-accretive on the space $X=L^1(\ren)$. Therefore, we have that there is a unique mild solution $u$ to \eqref{Pf}, obtained as a limit of discrete approximate solutions by the ITDS scheme. Moreover, inequality \eqref{L1_contr} follows.
\end{proof}

We concentrate next in the question of boundedness that will be a consequence of yet another feature of the theory that is important in itself, i.e., symmetrization.\nc

\begin{remark}\label{RemarkIp}
This section also holds under assumption  (H2) and $p_i>1$ making minor changes in the proof of Theorem \ref{ellipticexi}.
\end{remark}


\section{Symmetrization. New comparison results}\label{sec.symm}

In this section we assume that (H2) holds. We want to prove a comparison result based on Schwarz symmetrization. We start by considering the simpler setting of nonlinear elliptic equations posed in a bounded open set of
$\mathbb{R}^N$ with Dirichlet boundary condition  following the classical
paper \cite{Tal76}).  In our case, it is known that if $u$ solves the following stationary anisotropic problem in a bounded domain $\Omega$
\begin{equation}\label{P intro}
\ \left\{
\begin{array}
[c]{ll}%
-\underset{i=1}{\overset{N}{%
{\displaystyle\sum}
}}\left( \left\vert u_{x_{i}}\right\vert ^{p_{i}-2}u_{x_{i}%
}\right)  _{x_{i}}=f\left(  x\right)  & \mbox{  in }\Omega\\[6pt]
u=0 & \mbox{ on }\partial\Omega,
\end{array}
\right.
\end{equation}
then rearrangement methods allow to obtain a pointwise comparison result
for $u$ with respect to the solution of the suitable radially symmetric problem. Thus, in the case of energy solutions when the datum $f$ belongs to the dual space, it is proved in \cite{Cianchi} that if  $\Omega^{\sharp}$ is the ball centered in the origin such
that $|\Omega^{\sharp}|=|\Omega|$ and if $u^{\sharp}$ is the symmetric decreasing
rearrangement of a solution $u$ to problem (\ref{P intro}) then the following holds
 \begin{equation}
u^{\#}\leq U  \text{ \ \
\ \ \ in
\ }\Omega^{\#} \label{comparison}.%
\end{equation}
Here,  $U$ is the radially symmetric solution to the following isotropic problem:
\begin{equation}
\!\!\left\{  \!\!%
\begin{array}
[c]{ll}%
\Lambda\Delta_{\bar{p}}U=f^{\#}\left(
x\right)  \! &
\mbox{  in  }\Omega^{\#}\!\!\vspace{0.2cm}\\
\!\!U=0 & \!\mbox{ on }\partial\Omega^{\#}\!\!,
\end{array}
\!\!\right.  \label{Prob simm intro}%
\end{equation}
where $\overline{p}$ is the
harmonic mean of exponents $p_{1},\dots,p_{N}$, given by formula \eqref{p
bar}, while $f^{\#}$
the symmetric decreasing rearrangement of $f$. The result needs a constant $\Lambda>0$ that has been determined as
\begin{equation}
\Lambda=\frac{2^{\overline{p}}\left(  \overline{p}-1\right)  ^{\overline{p}%
-1}}{\overline{p}^{\overline{p}}}\left[
\frac{\underset{i=1}{\overset{N}{\Pi
}}p_{i}^{\frac{1}{p_{i}}}\left(  p_{i}^{\prime}\right)  ^{\frac{1}%
{p_{i}^{\prime}}}\Gamma(1+1/p_{i}^{\prime})}{\omega_{N}\Gamma(1+N/\overline
{p}^{\prime})}\right]  ^{\frac{\overline{p}}{N}}
\label{lambda}%
\end{equation}
with $\omega_{N}$ the measure of the $N-$dimensional unit ball,
$\Gamma$ the Gamma function, and
$p_{i}^{\prime}=\frac{p_{i}}{p_{i}-1}$ with the usual conventions if
$p_{i}=1$.

We stress that in contrast to the isotropic $p$-Laplacian equation,  not
only the space domain and the data of problem (\ref{P intro}) are
symmetrized with respect to the space variable, but also the
ellipticity condition is subject to an appropriate symmetrization.
Indeed, the diffusion operator in problem (\ref{Prob simm intro}) is the
standard isotropic $\overline{p}-$Laplacian.

\medskip

\subsection{ Main ideas of the parabolic symmetrization}
Now, it is well-known that a the pointwise comparison
(\ref{comparison}) need not hold for nonlinear parabolic equations, not even for the heat equation, and has to be replaced by a comparison of integrals known in the literature as Concentration Comparison, and reads (see
\cite{AdF, Bandle, Va82, Va82b, Va2005})
\begin{equation}
\int_{0}^{s}u^{\ast}(\sigma,t)\;d\sigma\leq\int_{0}^{s}U^{\ast}(\sigma
,t)\;d\sigma\qquad\hbox{in $(0, |\Omega|)$,} \label{integral_into}%
\end{equation}
valid for all fixed $t\in\left(0,T\right)$. Here, $u^{\ast}$ is the one dimensional, decreasing rearrangement
with respect to the space variable of the  weak energy  solution $u$ to the following problem
\begin{equation*}
\left\{
\begin{array}
[c]{ll}%
u_t-\sum_{i=1}^N (|u_{x_i}|^{p_i-2}u_{x_i})_{x_i}=f(x,t) & \quad
\hbox{in $\Omega\times (0, T)$}\\
\\
u(x,t)=0 & \quad\hbox{on $\partial \Omega\times (0, T),$}\\
\\
u(x,0)=u_{0}(x) & \quad\hbox{in $\Omega$}
\end{array}
\right.  \label{Pr}%
\end{equation*}
when the datum belongs to the dual space \nlc and $U^{\ast}$ is the same type of rearrangement of the solution $U$ to the following isotropic
"symmetrized" problem
\begin{equation*}
\left\{
\begin{array}
[c]{ll}%
U_{t}-\Lambda\Delta_{\bar{p}}U =f^{\#}(x,t) & \quad
\hbox{in $\Omega^{\#}\times (0, T)$}\\
\\
U(x,t)=0 & \quad\hbox{on $\partial \Omega^\#\times (0, T)$}\\
\\
U(x,0)=u_{0}^{\sharp}(x) & \quad\hbox{in $\Omega^{\#}$},\\
\end{array}
\right.  \label{symm Pr}%
\end{equation*}
respectively, { with $\Lambda$ defined in \eqref{lambda}, $u_{0}^{\#}$ the symmetric decreasing rearrangement of $u_0$  and $f^{\#}(x,t)$ the symmetric decreasing rearrangement of $f$ with respect to $x$ for $t$
fixed.}

\medskip

Let $u$
be a measurable function on $\R^{N}$ (if $u$ is defined on a bounded domain $\Omega$, we extend $u$ by 0 outside $\Omega$)
fulfilling
\begin{equation*}
\left\vert \{x\in\mathbb{R}^{N}:\left\vert u(x)\right\vert
>t\}\right\vert
<+\infty\text{ \ \ \ for every }t>0. \label{insieme livello di misura finita}%
\end{equation*}
The (Hardy Littlewood) one dimensional \textit{decreasing rearrangement} $u^{\ast}$ of $u$ is defined
as
\[
u^{\ast}(s)=\sup\{t>0:\left\vert \{x\in{\ren}:\left\vert u(x)\right\vert
>t\}\right\vert>s\}\text{ \ for }s\geq0,
\]
and the \textit{symmetric decreasing rearrangement} of $u$ is the
function $u^{\#}:\mathbb{R}^{N}\rightarrow\left[0,+\infty\right[$ given by

\[
u^{\#}(x)=u^{\ast}(\omega_{N}\left\vert x\right\vert
^{N})\text{
\ \ }\hbox{\rm for a.e.}\;x\in{{\mathbb{R}}^{N}.}%
\]

In what follows we need the following order relationship, taken from \cite{Va82}. Given two radially symmetric functions $f, g\in L^1_{loc}(\mathbb{R}^N)$ we say
that $f$ is more concentrated than $g$, $f\succ g$, if for every $R>0$,
$$\int_{B_R(0)}
f(x) dx \geq \int_{B_R(0)}
g(x) dx.$$

\subsection{Comparison result for stationary problems in the whole space with a lower-order term}

A lack of pointwise comparison already arises in elliptic equations with lower order terms, which have a close relationship with parabolic equations. See  in this respect \cite{Va2005} where the isotropic case is treated. Indeed, by the Crandall-Liggett implicit discretization scheme   \cite{CL71} (see below or \cite{Vlibro}), the parabolic comparison can be obtained from  a similar comparison result for the following stationary problem with a lower-order term:
\begin{equation}\label{Pf RN}
\left\{
\begin{array}
[c]{ll}%
\sum_{i=1}^N (|u_{x_i}|^{p_i-2}u_{x_i})_{x_i}+ \mu u=f & \quad
\hbox{in $\mathbb{R}^N$}\\[6pt]
{u(x)\rightarrow0} & \quad \hbox{ as } |x|\rightarrow\infty
\end{array}
\right.
\end{equation}
for arbitrary $\mu>0$.

\begin{theorem}\label{Thm Comp bis}
Let $u$ be the solution of problem \eqref{Pf RN} with $f\in L^1(\mathbb{R}^N)$ and let
$U$ be the solution of the following isotropic problem
\begin{equation}\label{Pf RN sim}
\left\{
\begin{array}
[c]{ll}%
-\Lambda\Delta_{\bar{p}}U +\mu U=g
 & \quad
\hbox{in $\mathbb{R}^N$}\\[6pt]
{u(x)\rightarrow0} & \quad \hbox{ as } |x|\rightarrow\infty
\end{array}
\right.
\end{equation}
with $g=g^\#$. If
$f^\#\prec g,$
then we have \ $u^\#\prec U$.
\end{theorem}
\noindent {\sl Proof.}
We can argue as in Theorem 3.6 of \cite{AdF} but considering the problem defined in whole space $\mathbb{R}^N$ and with a smooth datum. In order to obtain the result when the datum is in $L^1(\mathbb{R}^N)$ we argue by approximation (see section \ref{existence}) and we pass to the limit  in the concentration estimate, recalling that the rearrangement application $u\rightarrow u^*$ is a contraction in $L^r(\mathbb{R}^N)$ for any $r\geq1$ (see \cite{K2006}).
 \qed

\subsection{Statement and proof of the parabolic comparison result}

Now we are in position to state a comparison result for problem  \eqref{Pf}. We set $Q:=\R^{N}\times (0,\infty)$.

\begin{theorem}\label{Thm Comp}
Let $u$ be the mild solution of problem \eqref{Pf}  with initial data $u_0\in L^1(\mathbb{R}^N)$ and $f\in L^1(Q)$. Let
$U$ be the mild solution to the isotropic parabolic problem
\begin{equation} \label{eqcauchysymm}
\left\{
\begin{array}
[c]{lll}%
U_t-\Lambda\Delta_{\bar{p}}U=g  &  & \text{in }Q
\\[6pt]
U(x,0)=U_{0}(x) &  & x\in\R^{N}\,
\end{array}
\right. %
\end{equation}
with a nonnegative rearranged initial datum $U_0\in L^{1}(\R^{N})$  and nonnegative source $g\in L^{1}(Q)$ which is rearranged w.r. to $x\in \R^{N}$. Assume moreover that
\begin{itemize}
\item[i)] $u_0^\#\prec U_0$,
\item[ii)] $f^\#(\cdot, t)\prec g(\cdot,t)$ for every $t\geq0$.
\end{itemize}
Then, for every $t\geq0$
\begin{equation}\label{conc.comp}u^\#(\cdot, t) \prec U(\cdot, t).
\end{equation}
In particular, for every $q \in[1,\infty]$ we have comparison of $L^q$ norms,
\begin{equation}\label{norm comp}
\|u(\cdot,t)\|_q \leq\|U(\cdot,t)\|_q
\end{equation}
\end{theorem}

Note that the norms of \eqref{norm comp}
can also be infinite for some or all values of $q$.

\begin{proof}
According to what explained in Theorem \ref{parabolicexistence}, we use the implicit time discretization scheme to obtain the mild solutions to the parabolic problems. For each time $T>0$, we divide the time interval $[0,T]$ in $n$ subintervals
$(t_{k-1},t_{k}]$, where $t_{k}=kh$ and $h=T/n$ and we perform a discretization of $f$ and $g$ adapted to the time mesh $t_k=kh$, let us call
them $\{f_k^{(h)}\},\,\{g_k^{(h)}\} $, so that the piecewise constant (or
linear in time) interpolations of this  sequences give the function $f^{(h)}(x,t)$, $g^{(h)}(x,t)$ such that
$\|f-f^{(h)}\|_1\to 0$ and $\|g-g^{(h)}\|_1\to 0$  as $h\to 0$. We can define $f_k^{(h)},\,g_k^{(h)} $ in this way
\[
f_k^{(h)}(x)=\frac{1}{h}\int_{(k-1)h}^{kh}f(x,t)dt,\quad g_k^{(h)}(x)=\frac{1}{h}\int_{(k-1)h}^{kh}g(x,t)dt.
\]
Now we construct the function $u_{h}$ which is piecewise constant in each
interval
$(t_{k-1},t_{k}]$,
by
\[
u_{h}(x,t)=
\left\{
\begin{array}
[c]{lll}%
u_{h,1}(x)  &  & if\,\,t\in[0,t_{1}]%
\\[6pt]
u_{h,2}(x)  &  & if\,\,t\in(t_{1},t_{2}]
\\
[6pt]
\cdots
\\
[6pt]
u_{h,n}(x)  &  & if\,\,t\in(t_{n-1},t_{n}]
\end{array}
\right. %
\]
where $u_{h,k}$ solves the equation
\begin{equation}
h\mathcal{A}(u_{h,k})+u_{h,k}=u_{h,k-1}+f_k^{(h)}\label{eq.18}
\end{equation}
with the initial value $u_{h,0}=u_{0}$.
Similarly, concerning the symmetrized problem \eqref{eqcauchysymm}, we define the piecewise constant function $U_{h}$ by \[
U_{h}(x,t)=
\left\{
\begin{array}
[c]{lll}%
U_{h,1}(x)  &  & if\,t\in[0,t_{1}]%
\\
[6pt]
U_{h,2}(x)  &  & if\,t\in(t_{1},t_{2}]
\\
[6pt]
\cdots
\\
[6pt]
U_{h,n}(x)  &  & if\,t\in(t_{n-1},t_{n}]
\end{array}
\right. %
\]
where $U_{h,k}(x)$ solves the equation
\begin{equation}
-h\Delta_{\bar{p}}U_{h,k}+U_{h,k}=U_{h,k-1}+g_k^{(h)}\label{eq.20}
\end{equation}
with the initial value $U_{h,0}=U_{0}$.
Our goal is now to compare the solution $u_{h,k}$ to \eqref{eq.18} with the solution \eqref{eq.20} by means of mass concentration comparison. We proceed by induction. Using Theorem
\ref{Thm Comp bis}, we get
\[
u^{\#}_{h,1}\prec U_{h,1}.
\]
If we assume by induction that $u^{\#}_{h,k-1}\prec U_{h,k-1}$ and call $\widetilde{u}_{h,k}$ the (radially decreasing) solution to the
equation
\begin{equation*}
h\mathcal{A}(\widetilde{u}_{h,k})+\widetilde{u}_{h,k}=u^{\#}_{h,k-1}+(f_k^{(h)})^{\#},
\end{equation*}
Theorem \ref{Thm Comp bis} again implies
\begin{equation}\label{massconcdiscrete}
u^{\#}_{h,k}\prec \widetilde{u}_{h,k}\prec U_{h,k}\,,
\end{equation}
hence \eqref{massconcdiscrete} holds for all $k=1,\ldots,n$. Hence the definitions of $u_{h}$ and $U_{h}$ immediately imply
\begin{equation}
u_{h}(\cdot,t)^{\#}\prec U_{h}(\cdot,t))\label{massconcparab}
\end{equation}
for all times $t$.
Since we have
\[
u_{h}\rightarrow u,\quad U_{h}\rightarrow U\,\, \text{uniformly},
\]
passing to the limit in \eqref{massconcparab} we get the result.
\end{proof}


\section{Boundedness of solutions.}\label{sec.bdd}

This result is usually known as the $L^1$-$L^\infty$ smoothing effect. We
assume conditions (H2) and (H3).

\begin{theorem}\label{L1LI}  If $u_0\in L^1(\mathbb{R}^N)$, then the mild
solution to \eqref{APL} with initial condition \eqref{IC} satisfies the $L^\infty$ bound:
 \begin{equation}\label{Linfty-L1}
 \|u(t)\|_\infty\leq C t^{-\alpha}\|u_0\|_1^{\bar{p}\alpha/N}\quad \forall t>0,
 \end{equation}
where the exponent $\alpha$ is just the one defined in \eqref{alfa} and $C=C(N,\bar p)$.
 \end{theorem}

\noindent {\sl Proof.}
It is clear that the worst case with respect to the symmetrization and concentration comparison in the class of solutions with the same initial mass $M$ is just the
Barenblatt solution $B$ of the isotropic $\bar p-$laplacian with Dirac mass initial data, \emph{i.e.} $u_0(x)=M\delta(x)$. We are thus reduced to calculate the $L^\infty$ norm of $B$:
\begin{equation*}
\|B\|_\infty=C(N,\overline{p})t^{-\alpha}\|u_0\|_1^{\bar{p}\alpha/N}.
\end{equation*}
Actually, there is a difficulty in taking $B$ as a worst case in the comparison,
namely that $B(x,0)$ is not a function but a Dirac mass. We overcome the difficulty by approximation.
We take first a solution with bounded initial data, $u_0\in L^1(\mathbb{R}^N)\cap L^\infty(\mathbb{R}^N)$. We then replace $B(x,t)$ by a slightly delayed function $B(x,t+\tau)$, which is a solution with initial data $B(x,\tau)$, bounded but converging to $M\delta(x)$ as $\tau\rightarrow0$. It is
then clear that for a small $\tau>0$ such solution is more concentrated than $u_0$. From the comparison theorem we get
$$
|u(x,t)|\leq \|B(\cdot,t+\tau)\|_{\infty} = C(N,\bar p)\,M^{\bar{p}\alpha/N}\,(t+\tau)^{-\alpha}
$$
which of course implies \eqref{Linfty-L1}. The result for general $L^1$ data follows by approximation and density once it is proved for bounded $L^1$ functions.
 \qed

\medskip

\noindent {\bf Remark.}\label{Remark boundness} From Proposition \eqref{decay of the $L^p$} and Theorem \eqref{L1LI} we have that for $u_0\in L^1\cap L^\infty$, the rescaled evolution solution $v$ \eqref{NewVariables} is uniformly bounded in time.



\section{Anisotropic upper barrier construction}\label{sec.upp}

The construction of an upper barrier in an outer domain will play a key role in the proof of existence of the fundamental solution in Section \ref{sec.ex.ssfs}. We assume (H1) and (H2) hold as in the Introduction.

\begin{proposition}\label{P1}
The function
\begin{equation}\label{outer.barr}
\overline{F}(y)= \left( \sum_{i=1}^N \gamma_i|y_i|^{\frac{p_i}{2-p_i}} \right)^{-1}
\end{equation}
with
\begin{equation}\label{gammai}
\gamma_i
\leq\left[\frac{\alpha}{N}\left(\min_i\{\sigma_i\frac{p_i}{2-p_i}\}-1\right)\frac{1}{2(p_i-1)}
\left(\frac{p_i}{2-p_i}\right)^{-p_i}\right]^{\frac{1}{2-p_i}}
\end{equation}
is a weak supersolution to \eqref{StatEq} in ${\mathbb{R}^N\setminus B_R(0)}$ and a classical supersolution  in $\mathbb{R}^N\setminus\left\{0\right\}$, with $B_R(0)$ being a  ball of radius $R>0$. Moreover, $\overline{F}\in L^1(\mathbb{R}^N\setminus B_R(0))$ .
\end{proposition}

\noindent {\sl Proof of Proposition \ref{P1}}.
We observe that from our hypotheses $1<p_i<2$ and (H2)  and the value of
$\alpha$ and $\sigma_i$ guarantee that
\begin{equation}\label{33}
\frac{2-p_i}{p_i}<\sigma_i,
\end{equation}
that gives the summability outside a ball centered in the origin (see \cite[Lemma 2.2]{SJ05}). Note that $p_i/(2-p_i)\geq1$.
Let $\gamma_i$ some positive constants that we will choose later. Denoting $
X=\sum_{j=1}^N\gamma_j|y_j|^{p_j/(2-p_j)}$, for $y\in\mathbb{R}^N\setminus\cup_{i=1}^N\{y\in \mathbb{R}^N:y_i=0\}$ we have
\begin{equation*}
\begin{split}
I:=&\sum_{i=1}^{N}\left[(|\overline{F}_{y_i}|^{p_i-2}\overline{F}_{y_i})_{y_i}+a_i\left( y_i \overline{F}\right)_{y_i}
\right]
\\
&\leq
\sum_{i=1}^{N} 2(p_i-1)\left(\frac{p_i\gamma_i}{2-p_i}\right)^{p_i}X^{-2p_i+1}|y_i|^{2p_i\frac{p_i-1}{2-p_i}}
+\alpha X^{-1}-X^{-2}\sum_{i=1}^{N}\alpha \sigma_i\gamma_i\frac{p_i}{2-p_i}|y_i|^{\frac{p_i}{2-p_i}}
\\
&=X^{-1}\left[
\sum_{i=1}^{N} 2(p_i-1)\left(\frac{p_i\gamma_i}{2-p_i}\right)^{p_i}X^{-2p_i+2}|y_i|^{2p_i\frac{p_i-1}{2-p_i}}
+\alpha -X^{-1}\sum_{i=1}^{N}\alpha \sigma_i\gamma_i\frac{p_i}{2-p_i}|y_i|^{\frac{p_i}{2-p_i}}
\right]
\\
&\leq
X^{-1}\left[
\sum_{i=1}^{N} 2(p_i-1)\left(\frac{p_i\gamma_i}{2-p_i}\right)^{p_i}X^{-2(p_i-1)}|y_i|^{2p_i\frac{p_i-1}{2-p_i}}
+\alpha\left(1-\min_i\{\sigma_i\frac{p_i}{2-p_i}\}\right)
\right]
\end{split}
\end{equation*}
Since for every $i$ we have $$\gamma_i|y_i|^{p_i/(2-p_i)}\leq\sum_{j=1}^N\gamma_j|y_j|^{p_j/(2-p_j)}=X,$$ it follows that $$X^{-2(p_i-1)}\leq\gamma_i^{-2(p_i-1)}|y_i|^{2p_i(1-p_i)/(2-p_i)},$$ then
\begin{equation*}
I
\leq
X^{-1}
\sum_{i=1}^{N} \left[2(p_i-1)\left(\frac{p_i}{2-p_i}\right)^{p_i}\gamma_i^{2-p_i}
+\frac{\alpha}{N}\left(1-\min_i\{\sigma_i\frac{p_i}{2-p_i}\}\right)\right],
\end{equation*}
where $1-\min_i\{\sigma_i\frac{p_i}{2-p_i}\}<0$ by \eqref{33}. In order to conclude that $I\leq0$ it is enough to show that
\begin{equation*}
2(p_i-1)\left(\frac{p_i}{2-p_i}\right)^{p_i}\gamma_i^{2-p_i}
+\frac{\alpha}{N}\left(1-\min_i\{\sigma_i\frac{p_i}{2-p_i}\}\right)\leq0
\end{equation*}
for every $i=1,..,N$, i.e. \eqref{gammai}. {It is easy to check that computations works  for $y\in\mathbb{R}^N\setminus\left\{0\right\}$}. Finally we stress that
$\overline{F}_{y_i}\in L^{p_i}(\mathbb{R}^N\setminus B_R(0))$ with $R>0$
and then we can easy conclude that $\overline{F}$ is a weak super-solution as well.
\qed

{
\begin{remark}
We stress that $\overline{F}$ is a weak supersolution to \eqref{StatEq} in $\mathbb{R}^N\setminus \{\sum_{j=1}^N\gamma_j|y_j|^{p_j/(2-p_j)}\leq R\}$ and  belongs to $L^1(\mathbb{R}^N\setminus\{\sum_{j=1}^N\gamma_j|y_j|^{p_j/(2-p_j)}\leq R\})$ for any $R>0$. Moreover if $F_*$ is the value
of $\overline{F}$ on $\{\sum_{j=1}^N\gamma_j|y_j|^{p_j/(2-p_j)}=1/F_*\}$, then $\min\left\{\overline{F},F_*\right\}$ agrees with $\overline{F}$ on $\{\sum_{j=1}^N\gamma_j|y_j|^{p_j/(2-p_j)}\geq 1/F_*\}$ and with $F_*$ on $\{\sum_{j=1}^N\gamma_j|y_j|^{p_j/(2-p_j)}< 1/F_*\}$.
\end{remark}
}

We are ready to prove a comparison theorem that is needed in the proof of
existence of the self-similar fundamental solution. We set as a barrier the truncation of the supersolution $\overline{F}(y)$ given in \eqref{outer.barr}. The proof is similar to \cite[Theorem 3.2]{FVV2020} but for the sake of completeness we include here the details.

\begin{theorem}[Barrier comparison]\label{thm.barr} For any $M>0$ and $L_1>0$, there exists $F_*$ such that if $v_0(y)\ge 0$ is a $L^1$ bounded function such that imposing
\begin{itemize}
\item[(i)] $v_0(y)\le L_1$  a.e. in $\ren$
\item[(ii)] $\int v_0(y)\,dy\le M$
\item [(iii)]  $v_0(y)\le G_{M,L_1}(y)$ a.e in $\ren$
\end{itemize}
where $G_{M,L_1}=\min\left\{\overline{F},F_*\right\}$ is the truncation
of $\overline{F}(y)$ given in \eqref{outer.barr} at level $F_*$,
then
\begin{equation}\label{outer.comp}
v(y,\tau)\le G_{M,L_1}(y) \quad \mbox{ for } a.e. \ y\in \ren, \ \tau>\tau_0.
\end{equation}
where $v(y,\tau)$ solves \eqref{APLs} with initial datum $v_0(y)$.
\end{theorem}

\noindent {\sl Proof.}
(i)   Let us pick some $\tau_1>0$. Starting from initial
mass $M>0$, from the smoothing effect \eqref{Linfty-L1} and the scaling transformation \eqref{NewVariables}(we put $t_0=1$ and then $\tau_0=0$), we know that
\begin{equation}
v(y,\tau)= (t+1)^{\alpha}u(x,t)\le C_1 M^{{\bar{p}\alpha/N}}((t+1)/t)^{\alpha}= C_1 M^{{\bar{p}\alpha/N}}(1-e^{-\tau})^{-\alpha},\label{firstsmooth}
\end{equation}
where $C_1$ is an universal constant as in \eqref{Linfty-L1}. Since $\tau=\log (t+1)$ we have $\|v(\tau)\|_\infty\le F_*$ for all $\tau\ge \tau_1$ if $F_*$ is such that
\begin{equation}\label{barr.cond1}
C_1 M^{{\bar{p}\alpha/N}}(1-e^{-\tau_1})^{-\alpha}\le F_*.
\end{equation}

(ii) For $0\leq\tau<\tau_1$ we argue as follows: from $v_0(y)\le L_1$ a.e. we get $u_0(x)\le L_1$ a.e., so $u(x,t)\le L_1$ a.e., therefore
$$
\|v(\tau)\|_\infty \le L_1 (t+1)^{\alpha}= L_1 e^{\alpha\tau} \quad a.e. .
$$
We now impose $F_*$ is such that
\begin{equation}\label{barr.cond3}
L_1 e^{\alpha\tau_1}\le F_*.
\end{equation}
Then we choose $F_*$ such that \eqref{barr.cond1} and \eqref{barr.cond3} hold.

(iii) Under these choices  we get $\|v(\tau)\|_\infty\le F_*$  for every $\tau>0$, which gives a comparison between  $v(y,\tau)$ with $G_{M,L_1}(y)$ in the complement of the exterior cylinder $Q_o=\Omega\times (0,\infty)$, where $\Omega=\{y:\overline F\leq F_*\}$, i.e. ${\{\sum_{j=1}^N\gamma_j|y_j|^{p_j/(2-p_j)}\geq 1/F_*\}}$. By the comparison in Proposition \ref{Prop ComU} for solutions in  $Q_o$ we conclude that
$$
v(y,\tau)\le G_{M,L_1}(y) \quad \mbox{ for }a.e. \ y\in \Omega, \ \tau>0,
$$
The comparison for $y\not\in \Omega$ has been already proved, hence the result \eqref{outer.comp}.
\qed




 As a consequence of mass conservation and the existence of the upper barrier we obtain a positivity lemma for certain solutions of the equation. This is the uniform positivity that is needed in the proof of existence of self-similar solutions, and it avoids the  fixed point from being trivial.

\begin{lemma}[A quantitative positivity lemma]\label{lem.pos}
Let $v$ be the solution of the rescaled equation \eqref{APLs} with integrable initial data $v_0$ such that: $v_0$ is a SSNI, bounded, nonnegative function with support in the ball of radius $R$,  $\int v_0(y)\,dy=M>0$
and  $v_{0}\leq G_{M,L_1}$ a.e., where $G_{M,L_1}$ is as in Theorem \ref{thm.barr}. Then, there is a continuous nonnegative function $\zeta(y)$, positive in a ball of radius $r_0>0$, such that
\begin{equation*}
v(y,\tau)\ge  \zeta(y) \quad \mbox{ for } {a.e.}  \ y\in \ren, \ \tau>0.
\end{equation*}
In particular, we may take $\zeta(y)\ge c_1>0$ a.e. in $ B_{r_0}(0)$ for suitable $r_0$ and $c_1>0$. The function $\zeta$ will depend on the choice of $M$ and $\|v_{0}\|_{\infty}$.
\end{lemma}

We will recall the denomination SSNI stands for separately symmetric and nonincreasing. It was introduced in \cite{FVV2020}. The proof of Lemma \ref{lem.pos} runs as Lemma 5.1 of \cite{FVV2020}.

\section{Existence of a self-similar fundamental solution}\label{sec.ex.ssfs}

Now we are ready to prove the main Theorem of this Section, dealing with the difficult problem of finding a self-similar fundamental solution to \eqref{APL}, enjoying good symmetry properties and the expected decay rate
at infinity.
\begin{theorem}\label{thm.fundamental solution}
For any mass $M>0$ there is a self-similar fundamental solution of equation \eqref{APL} with mass $M$. The profile $F_M$ of such solution is a SSNI nonnegative function. Moreover $F_M(y)\le \overline{F}(y)$ for a.e. $y$
such that $|y|$ big enough, where  $\overline{F}(y)$ is given in \eqref{outer.barr}.
\end{theorem}

\noindent{\bf Remark.} Therefore, we get an upper bound for the behaviour
of $F$ at infinity. It has a clean  form in every coordinate direction: $F(y)\le O(|y_i|^{-\frac{p_i}{(2-p_i)}})$ as $|y_i|\to\infty$.

The basic existence with self-similarity follows as in Theorem 6.1 of \cite{FVV2020}. The full existence includes self-similarity and will be established next.


\subsection{Proof of existence of a self-similar solution}\label{existself}

We will proceed in a number of steps.

\noindent (i) Let $\phi\geq 0$ be bounded, symmetric decreasing with respect to $x_{i} $, supported in a ball of radius 1 centered at 0, with total mass $M$ (we ask for such specific properties for convenience).
We consider the solution $u_1$ such that $u_1(x,1)=\phi$, which is bounded and integrable for all $t>1$,  and denote
\begin{equation}\label{uk}
u_k(x,t)=\mathcal{T}_ku_1(x,t)=k^{\alpha}u_1(k^{\sigma_1\alpha}x_1, ..., k^{\sigma_N\alpha}x_N,kt)
\end{equation}
for every $k>1$. We want to let $k\rightarrow\infty$. In terms of rescaled variables {\eqref{NewVariables}}(with $t_0=0$) we have
\begin{equation*}
\begin{split}
v_k(y,\tau)&=e^{\alpha\tau }u_k(y_1e^{\alpha \sigma_1 \tau},...,y_Ne^{\alpha \sigma_N \tau}, e^\tau)
\\
&=e^{\alpha \tau }k^{\alpha}
u_1(k^{\sigma_1\alpha}y_1e^{\tau\sigma_1\alpha}, ..., k^{\sigma_n\alpha}x_Ne^{\tau\sigma_N\alpha}, ke^\tau),\nc
\end{split}
\end{equation*}
where $t=e^{\tau}$, $\tau>0$.  Put $k=e^{h}$ so that $k^{\sigma_i\alpha}e^{\tau\sigma_i\alpha}=e^{(\tau+h)\sigma_i\alpha}$. Then,
\begin{equation*}
v_k(y,\tau)
=e^{(\tau+h) \alpha }
u_1\left(y_1e^{(\tau+h)\sigma_1\alpha}, ..., y_Ne^{(\tau+h)\sigma_N\alpha},e^{(\tau+h)}\right).
\end{equation*}
Putting $v_1(y',\tau')=t^{\alpha}u_1(x,t)$ with $y'_i=x_i\,t^{-\alpha\sigma_i}$, $\tau'=\log t$, then
 \begin{equation*}
v_k(y,\tau)=
e^{(\tau+h-\tau') \alpha }
v_1(y_1e^{(\tau+h-\tau')\sigma_1\alpha}, ..., y_Ne^{(\tau+h-\tau')\sigma_N\alpha},\tau+h).
\end{equation*}
Setting $\tau'=\tau+h$, we get
\begin{equation}\label{vk}
v_k(y,\tau)=v_1(y,\tau+h).
\end{equation}
This means that the transformation $\mathcal{T}_k$ becomes a forward time
shift in the rescaled variables that we call $\mathcal{S}_h$ with $h=\log k$.

\medskip

 (ii)  Next, we prove the existence of periodic orbits with the following
setup.  \nc We take  $X = L^1(\mathbb{R}^N)$ as ambient space and consider an important subset of $ X$ defined as follows.

For any $L_{1}>0$, we define the set $K=K(L_{1})$ as the set of all $\phi\in L_{+}^{1}(\R^{N})\cap L^{\infty}(\R^{N})$  such that: \

(a) $\int \phi(y)\,dy = 1,$

(b) $\,\phi\,$ is SSNI (separately symmetric and nonincreasing w.r.\,to all coordinates),

(c) $\phi$  is  a.e. bounded above by $G_{L_{1}}(y)$, being $G_{L_{1}}(y)=\min\left\{\overline{F},F_{\ast}\right\}$ a fixed barrier, with $F_{\ast}$ conveniently large and $\overline{F}(y)$ defined in \eqref{outer.barr},

(d) $\phi$  is uniformly bounded above by $L_1>0$.

Observe that $G_{L_{1}}(y)$ is obtained in Theorem \ref{thm.barr} by truncating $\overline{F}(y)$ at a convenient level $F_*$: this gives that $G_{L_{1}}(y)$ is a barrier for solutions to \eqref{APLs} with mass $M=1$ and initial data verifying the assumption of Theorem \ref{thm.barr}.

\medskip

\noindent By the previous considerations, it is easy to see that  $K(L_1)$ is a non-empty, convex, closed and bounded subset with respect to the norm of the Banach space $X$.

Now, for all $\phi\in K(L_1)$ we consider the solution $v(y,\tau)$ to equation \eqref{APLs} starting at $\tau = 0$ with data $v(y, 0) =\phi(y)$, and we consider for all small $h>0$ the semigroup map $S_h : X \rightarrow X$ defined by $\mathcal{S}_h(\phi) =v(y, h)$. The following lemma collects some facts we need.

\begin{lemma}\label{lemma.below}
Given $h>0$, there exists $L_{1}=L_{1}(h)$ such that $S_h(K(L_1(h))\subset K(L_1(h))$. Moreover, $S_h(K(L_1(h)))$  is relatively compact in $X$.
Finally, for every $\phi\in K(L_1(h))$
\begin{equation}\label{bdd.below}
S_\tau\phi(y)\ge \zeta{_h}(y)   \quad \mbox{ for } a.e. \ y\in \ren, \ \tau>0,
\end{equation}
where $\zeta_{h}$ is a fixed function as in Lemma \ref{lem.pos}. It only depends on $h$.
\end{lemma}

\noindent {\sl Proof.}
Fix a small $h>0$, and let $L_1=L_{1}(h)$ such that
\begin{equation}\label{L1 h}
L_{1}\geq C_{1}M^{{\bar{p}\alpha/N}}(1-e^{-h})^{-\alpha},
\end{equation}
where $C_{1}$ is the constant in the smoothing effect \eqref{Linfty-L1}. We take ${\color{magenta}\tau_1}=h$ in the proof of Theorem \eqref{thm.barr} and choose $F_{\ast}=F_{\ast}(h)$ such that \eqref{barr.cond3} holds, that is
\[
L_1 e^{\alpha h}\le F_*.
\]
Then we have in particular that \eqref{barr.cond1} is satisfied, namely
\[
C_1 M^{{\bar{p}\alpha/N}}(1-e^{-h})^{-\alpha}\le F_*.
\]
This ensures the existence of a barrier $G_{L_{1}(h)}(y)$ (a truncated of
$\bar F$ defined in \eqref{outer.barr}), such that for $\phi\in K(L_1(h))$ and any $\tau>0$ we have $\mathcal{S}_{\tau}(\phi)\leq G_{L_{1}(h)}(y)$
a.e.. Then $\mathcal S_{h}(\phi)$ obviously verifies (c), while (a) is a consequence of mass conservation and (b) follows by Proposition \ref{Prop
3}. Moreover, \eqref{L1 h} ensures that from \eqref{firstsmooth} we immediately find $\mathcal{S}_{h}(\phi)\le L_1 $ a.e., that is property (d).
The relative compactness comes from known regularity theory. The last estimate \eqref{bdd.below} comes from
Lemma \ref{lem.pos}, which holds once a fixed barrier is determined.  \qed

It now follows from the Schauder Fixed Point Theorem, cf. \cite{Evansbk},
Theorem 3, Section 9.2.2,
that there exists at least fixed point $\phi_h \in K(L_{1}(h))$, \textit{i.\,e.,} $ \mathcal{S}_h(\phi_h) = \phi_h$. Set $ \mathcal{S}_{\tau}(\phi_h)=:v_{h}(y,\tau)$, thus in particular $v_{h}(y,0)=\phi_{h}(y)$. The fixed point is in $K$, so it is not trivial because it has mass 1 and moreover it satisfies the lower bound \eqref{bdd.below}.
Iterating the equality, we get periodicity for the orbit $v_h(y, \tau)$ starting at $\tau = 0$
\begin{equation}\label{vh}
v_h(y,\tau+ kh) = v_h(y,\tau )\quad  \forall \tau > 0,
\end{equation}
this is valid for all integers $k\geq1$.

\medskip

(iii) Once the periodic orbit is obtained we may examine the family of periodic orbits $\{v_h: h>0\}$ as a way to obtain a stationary solution in the limit $h\to0$. Prior to that, let us derive a uniform boundedness property of this family  based on the rough idea that periodic solutions enjoy special properties. Indeed,  the smoothing effect  implies that any solution with mass $M\le 1$ will be bounded by $C_1\,t^{-\alpha}$ (see \eqref{Linfty-L1}) in terms of the $u$ variable, hence $v(y,\tau) $ will be bounded uniformly in $y$ for all large $\tau$ when written in the $v$ variable. Since our functions $v_h$ are periodic, this asymptotic property actually implies that each $v_h$ is a bounded function, uniformly in $y$ and $t$. On close inspection we see that the bound is also uniform in $h$, $v_h\le C_1$ a.e.. That is quite handy since then we can also get a positive lower bound  $\zeta$ valid for all times using uniform upper bounds in $L^\infty$, $L^1$ and the upper barrier $\overline{F}$. Then we have that the family $v_{h}$ is uniformly bounded in $L^{1}\cap L^{\infty}$, thus the family $v_{h}$ is equi-integrable. Moreover $v_{h}$ is tight, because the mass confinement holds: indeed, since $v_{h}\leq \overline{F}$ a.e. uniformly w.r. to $h$, for a large $R>0$ it follows that
\[
\int_{|y|>R}v_{h}\,dy<\int_{|y|>R}\overline{F}(y)\,dy,
\]
thus (recall that $\overline{F}\in L^{1}(\R^{N}\setminus B_{R}(0)))$
\[
\lim_{R\rightarrow\infty}\int_{|y|>R}v_{h}\,dy=0.
\]
 Then the Dunford-Pettis Theorem implies that, up to subsequences,
 \[
 v_{h}(\tau)\rightharpoonup \widehat v(\tau)\quad\text{weakly in }L^{1}(\mathbb{R}^N)
 \]
 for some $\widehat v(y,\tau)$. In particular, this gives $\|\widehat v(\tau)\|_{L^{1}}=1$. {Moreover, the a priori estimates \eqref{firstenergy},\eqref{esttimeder1},\eqref{esttimeder2} and the smoothing effect \eqref{Linfty-L1} allow to employ the usual compactness argument and find that $\widehat v$ solves the rescaled equation \eqref{APLs} in the limit.}
%
%
%
%
%
%
%

(iv) {We can now take the dyadic sequence $h_n=2^{-n}$ and $k_n=k'2^{n-m}$ with $n,m,k'\in\mathbb{N}$ and $m\leq n$  in this collection of  periodic orbits $v_h$. Inserting this values in \eqref{vh} and passing to the limit (along such subsequence) as $n\to \infty$, we find  the equality
$$
\widehat v(y,\tau+ k'2^{-m})=\widehat v(y,\tau) \quad \forall \tau>0
$$
holds for all integers $m,k'\ge 1$.} By continuity of the orbit in $L^1_{loc}$, $\widehat v$ must be stationary in time. Passing to the limit, we conclude that $\widehat v(y)\leq C$ and moreover $\widehat v(y)\leq \overline{F}$, which gives in particular the required asymptotic behaviour at infinity with the correct rate. Going back to the original variables, it means that the corresponding function $\widehat u(x,t)$ is a self-similar solution of equation \eqref{APL}. Hence, its initial data must be a non-zero Dirac mass. Now we choose any mass $M>0$. If $M=1$ then $\widehat u$ is the selfsimilar solution we looked for. If $M\neq1$,  we apply the mass changing scaling transformation \eqref{mass.trans}. \qed

 \begin{remark}(Local positivity) We know from the proof that $\widehat{v}(y)\leq C$ and $\widehat{v}(y)\leq \overline{F}$, then Theorem \eqref{thm.barr} and Lemma \eqref{lem.pos} ensures that $\widehat{v}(y)\geq \zeta(y)$ for some positive function $\zeta$. Hence, $\widehat{v}$ is locally positive.
 \end{remark}

We have a further property of the self-similar solutions that we will use
later.

\begin{proposition}\label{prop.ssni} Any non-negative self-similar solution $B(x,t)$ with finite mass is SSNI.
\end{proposition}

\noindent {\sl Proof}. We use two general ideas, (i) SSNI is an asymptotic property of many solutions, and (ii) self-similar solutions necessarily
verify asymptotic properties for all times.

Let  us consider a non-negative self-similar solution, $B(x,t)$. The issue is to prove it has the SSNI property. This is done by approximation and
rescaling. We begin with approximating $B$ at time $t=1$ with a sequence of bounded, compactly supported functions $u_n(x,1)$ with increasing supports and converging to $B(x,1)$ in $L^1(\mathbb{R}^N)$. We consider the corresponding solutions $u_n(x,t)$ to \eqref{APL}, for $t\ge 1$.

 The Aleksandrov principle says that these functions $u_n(\cdot,t)$ have as $t\to \infty$ an approximate version of the SSNI properties as follows. If the initial support at $t=1$ is contained in ball of radius $R>0$ then for all $t>1$ and for every $x,\widetilde{x}\in \ren$,  $|x|,$ $|\widetilde{x}|\ge 2R$, we have
\begin{equation}\label{alek.ssni}
u(x,t)\ge u(\widetilde{x},t)
\end{equation}
on the condition that $|\widetilde{x}^i|\ge |x^i|+2R$ for every $i=1,\cdots,N$. A convenient reference can be found in references like \cite{CVW87} or \cite[Proposition 14.27]{Vlibro}.

The last step is to translate this asymptotic approximate properties into
exact properties. This is better done in the $v$ formulation, introduce with formulas \eqref{NewVariables} and \eqref{APLs}. We first observe that
$u_{n}$ converges to some $\widetilde{B}$, thus by the contraction principle, for $t\geq1$
\[
\|u_{n}(t)-B(t)\|_{L^{1}(\R^{N})}\leq \|u_{n}(1)-B(1)\|_{L^{1}(\R^{N})}
\]
and passing to the limit as $n\rightarrow\infty$ we have $\widetilde{B}(x,t)=B(x,t)$ for $t\geq1$. This implies that the sequence $v_n(y,\tau)$ of rescaled solutions converges to the self-similar profile $F(x)=B(x,1)$ at $\tau\geq 0$ (i.e. $t\geq1$). 
On the other hand, the definition of the rescaled variables $y_i=x_i\,t^{-a_i}$ implies that the monotonicity properties derived for $u_n$ by Aleksandrov keep being valid in terms of $(y_1,\cdots, y_N)$ with the reformulation:
\begin{equation}\label{ssni}
v_{n}(y,\tau)\ge v_{n}(\widetilde{y},\tau)
\end{equation}
on the condition that $|\widetilde{y}_{i}|\ge |y_{i}|+2R\, t^{-a_i}$. Similarly, symmetry comparisons are true up to a displacement $R\,t^{-a_i}$.
Passing to the limit in \eqref{ssni} as $n\rightarrow\infty$, we find
\[
F(y)\geq F(\widetilde{y})
\]
provided $|\widetilde{y}_{i}|\ge |y_{i}|+2R\, t^{-a_i}$. Since $t$ can be
chosen arbitrarily large, the same property holds for $|\widetilde{y}_{i}|\ge|y_{i}|$. Thus $F$ is symmetric with respect to each $x_{i}$ and the full SSNI applies to $F$, hence to the original $B$.\qed

%
%

\medskip

\section{Lower barrier construction and global positivity}\label{sec.lbarr}

Now we get a  lower barrier that looks a bit like the upper barrier of Section \ref{sec.upp}.

\begin{proposition}\label{P2}
Let us take $\gamma>0$, let $0<\vartheta_i\leq1$ be chosen such that
\begin{equation}\label{gamma bis}
\frac{1}{\gamma\vartheta_i}<\frac{2-p_i}{p_i}(< \sigma_i).\end{equation}
Then,
\begin{equation}\label{F under}
\underline{F}(y)= \left(A+\sum_{i=1}^N|y_i|^{\vartheta_i}\right)^{-\gamma}
\in L^1(\mathbb{R}^N)
\end{equation}
is a weak sub-solution in $\mathbb{R}^N$ and a  classical  sub-solution to the stationary equation \eqref{StatEq} in
{$\mathbb{R}^N\setminus\cup_{i=1}^N\{y\in \mathbb{R}^N:y_i=0\}$} for $A>A_0$, where
\begin{equation}\label{r0under}
A_0:=\max_{i=N_0+1,\cdots,N}\left(\frac{N\gamma^{p_i-1} (p_i-1)(\gamma +1)\vartheta_i^{p_i}}{\alpha(\gamma\max_i\{\sigma_i\vartheta_i\}-1)}\right)
^{\frac{1}{\gamma-\gamma(p_i-1)-p_i/\vartheta_i}}.
\end{equation}
\end{proposition}

\medskip

\noindent {\sl Proof}.
Since $\vartheta_i\leq1$ we get
\begin{equation*}
\begin{split}
I:=&\sum_{i=1}^{N}\left[(|\underline{F}_{y_i}|^{p_i-2}\underline{F}_{y_i})_{y_i}+\alpha \sigma_i\left( y_i \underline{F}\right)_{y_i}
\right]
\\
&\geq \sum_{i=1}^{N}\left(A+\sum_{j=1}^N|\eta_j|^{\vartheta_j}\right)^{-(\gamma+1)(p_i-1)-1}\!\!\!\!\!\!\!\!\gamma^{p_i-1} (p_i-1)(\gamma+1)\vartheta_i^{p_i}|y_i|^{p_i(\vartheta_i-1)}
+\alpha\left(A+\sum_{i=1}^N|y_i|^{\vartheta_i}\right)^{-\gamma}
\\
&-\gamma \alpha \max_i\{\sigma_i\vartheta_i\}\left(A+\sum_{i=1}^N|y_i|^{\vartheta_i}\right)^{-\gamma-1}
\sum_{i=1}^{N}|y_i|^{\vartheta_i}
\\
\geq &
\sum_{i=1}^{N}\left(A+\sum_{j=1}^N|\eta_j|^{\vartheta_j}\right)^{-(\gamma+1)(p_i-1)-1}\!\!\!\!\!\!\!\!\gamma^{p_i-1} (p_i-1)(\gamma+1)\vartheta_i^{p_i}
\left(A+\sum_{j=1}^{N}|y_j|^{\vartheta_j}\right)^{p_i(1-1/\vartheta_i)}
\\
&+\alpha\left(A+\sum_{i=1}^N|y_i|^{\vartheta_i}\right)^{-\gamma}
-\gamma \alpha\max_i\{\sigma_i\vartheta_i\}\left(A+\sum_{i=1}^N|y_i|^{\vartheta_i}\right)^{-\gamma-1}
\left(A+\sum_{i=1}^{N}|y_i|^{\vartheta_i}\right).
\end{split}
\end{equation*}
Denoting $X=A+\sum_{j=1}^N|\eta_j|^{\vartheta_j}$, we obtain
\begin{equation*}
I\geq
\sum_{i=1}^{N}X^{-\gamma(p_i-1)-p_i/\vartheta_i}\left[
\gamma^{p_i-1} (p_i-1)(\gamma+1)\vartheta_i^{p_i}
+X^{-\gamma+\gamma(p_i-1)+p_i/\vartheta_i}
\frac{\alpha}{N}\left(1-\gamma\max_i\{\sigma_i\vartheta_i\}\right)\right].
\end{equation*}
We stress that \eqref{gamma bis} yields $1-\gamma\max\{\sigma_i\vartheta_i\}\leq0$
and  $-\gamma+\gamma(p_i-1)+p_i/\vartheta_i<0$.
In order to have  $I\geq0$ we have to require $X\geq A_0$.
Choosing $A>A_0$, it follows that $\underline{F}$ is a sub-solution to equation \eqref{StatEq} in $\mathbb{R}^N\setminus\{0\}$.
{It is easy to check that $\underline{F}\in L^1(\mathbb{R}^N)$ and $\underline{F}_{y_i}\in L^{p_i}(\mathbb{R}^N)$ for all $i$.
In order to prove that it is a weak solution in all $\mathbb{R}^N$, we have to multiply for a test function $\psi \in \mathcal{D}(\mathbb{R}^N)$, to integrate in $\mathbb{R}^N\setminus \cup_{i=1}^N\{y:|y_i|<\varepsilon\}$ for $\varepsilon>0$ and finally to estimate the boundary terms. We observe that for every $i=1,\cdots,N$
$$\left|\int_{\partial \{[-\varepsilon,\varepsilon]^N\}} \underline{F} \,
y_i \partial_{y_i}\psi\, d \sigma\right|\leq A^{-\gamma}\|\psi_{y_i}\|_{\infty}C(N)\varepsilon^{N+1}$$
and
$$\left|\int_{\partial \{[-\varepsilon,\varepsilon]^N\}}|\partial_{y_i}\underline{F}|^{p_i-2} \partial_{y_i}\underline{F} \partial_{y_i}\psi \, d \sigma\right|\leq
 A^{-(\gamma+1)(p_i-1)}\|\psi_{y_i}\|_{\infty}C(N)\varepsilon^{N+(\vartheta_i-1)(p_i-1)},$$
where $N+(\vartheta_i-1)(p_i-1)>0$ under our assumptions. Similar computations work for the other boundary terms. It is clear that all boundary terms go to zero when $\varepsilon \rightarrow0$.}
\qed

\begin{remark}
Under the assumption of Proposition \ref{P2} we have
\begin{equation}\label{U under}
\underline{U}(x,t)=
t^{-\alpha}\underline{F} (t^{-\alpha\sigma_i}x_1, \cdots,t^{-\alpha\sigma_i}x_N)
\end{equation}
 is a weak sub-solution to \eqref{APL} in $\mathbb{R}^N\times[0,\infty)$ such that $\underline{U}(x,t)\rightarrow \|\underline{F}\|_{L^1}\delta_0(x)\, \text{ as }t\rightarrow 0$
 in distributional sense. In particular, for every $x\ne 0$ we have
 \begin{equation}\label{lim}
 \lim_{t\to 0} \underline{U}(x,t)=0.
 \end{equation}
\end{remark}


We prove a comparison result from below. We take as comparison the two functions

(i) the self-similar solution in original variables (with $t_0=1$ for simplicity)
$$
B(x,t)=(t+1)^{-\alpha}F(x_1(t+1)^{-\alpha \sigma_1},\cdots, x_N(t+1)^{-\alpha \sigma_N}),
$$
with $\alpha$ and $\sigma_i$ as prescribed in \eqref{alfa} and \eqref{ai}, and

(ii) the function $\underline{U}(x,t) $  stated in \eqref{U under}, that depends on the parameter $A$.

\begin{theorem}[Lower Barrier comparison]\label{thm.lowbarr}
There is a time $\overline{t}>0$, a radius $R>0$ and a constant $A$ large
enough, such that for every $|x|\ge R$, $ 0\le t\le \overline{t}$
 we have
\begin{equation}\label{compp}
\underline{U}(x,t) \le B(x,t)\,.
\end{equation}
\end{theorem}%

The proof of the previous theorem is a simple comparison in an outer cylinder that runs as one of Theorem 7.4 in \cite{FVV2020}, {since the limit \eqref{lim} is uniform in $x$ a long as $|x|\geq R > 0$ for $t>0$ small enough.}

$R$. To proceed, we first use the fact that for some small $R>0$, $F$ is positive in the  ball $B_{2R}(0)$ by the qualitative lower estimate, $F(y)\ge \zeta(y)\ge c_1>0$ (see Lemma \ref{lem.pos}). We also know that $F\ge 0$ everywhere. On the other hand, we known from the previous argument that \eqref{lim}

\medskip

From Theorem \ref{thm.lowbarr} we derive the positivity for small times
of the self-similar fundamental solution determined in Theorem \ref{thm.fundamental solution}. Furthermore, we have the following:

\begin{corollary}\label{thm.lowbeh}  If $F$ is the profile of a self-similar solution there are constants $c_1,c_2>0$ such that
\begin{equation}\label{compbelow}
F(x) \ge c_1 \, \underline{F}(x_1 \, c_2^{\alpha \sigma_1}, \cdots, x_N c_2^{\alpha \sigma_N})
\end{equation}
for every $|x|\ge R$,  if $R>0$ and $A_2$ is large enough. In particular, the profile $F$ decays at most like \
$O(|x_i|^{-\vartheta_i\gamma})$ in any coordinate direction.
\end{corollary}

To prove the previous corollary it is enough to evaluate \eqref{compp} at
$t=\overline{t}$.

\noindent {\bf Remarks.} We can make this estimate on the decay rate  as close as we want to \ $O(|x_i|^{-p_i/(2-p_i)})$. In view of the already obtained upper bounds, these exponents are sharp.

\medskip

We can pass from the positivity of just the fundamental solution to the  strict positivity for general solutions. This uses a  variation of Theorem 7.6 in \cite{FVV2020} together with the positivity result for the solutions of the fractional $p$-Laplacian equation, which has been proved in \cite{VazFPL2-2020}, Section 6.

\begin{theorem}[Infinite propagation of positivity]\label{thm.genpos}
Any integrable  solution with continuous and nonnegative initial data and
positive mass is strictly  positive  a.e.  in $\ren\times (0,\infty)$.
\end{theorem}%

\noindent {\sl Proof.} (i) Arguing as in the proof of Theorem 7.6 in \cite{FVV2020} we obtain the infinite propagation of positivity of $u$ when the initial datum $u_0$ is SSNI, continuous and compactly supported.

%

(ii) Take now a  continuous  initial datum $u_0\ge 0$. We can put below $u_0$ a smaller SSNI continuous  compactly supported initial datum $\widetilde{u}(x)$ as in point (i) around some point $x_0$,   and in particular $u_0(x)\geq \widetilde{u}(x)$ in $\mathbb{R}^N$. If $u_{1}(x,t)$ is the solution of the Cauchy problem with data $\widetilde{u}$, we use the  Comparison principle to obtain that $u(x,t)\geq  u_1(x,t)>0$ a.e. in $\mathbb{R}^N$ for every $t>0$. Hence $u$ is strictly positive in $\R^{N}$ in the
sense of measure theory.\normalcolor
%
$t_0-\ve<t<t_0+t_2-\ve$. After checking that $t_2$ does not depend on $\ve$ we conclude that $u(x,t_0)>0$.
\nc


\section{The orthotropic case} \label{sec.ortho}

In this Section we consider the equation \eqref{APL} in the \emph{orthotropic} case, namely when all exponents are equal, $p_{1}=...=p_{N}=p<2$. We have to restrict ourselves to this case to prove a uniqueness result for SSNI fundamental solutions, because we need some solution regularity that has not yet been proved (to our knowledge) in the general anisotropic case.\nlc

\subsection{Continuity of solutions}\label{sec.cont}

This Subsection is devoted in proving the continuity of mild solutions to
the Cauchy problem for equation \eqref{APL}, in the orthotropic case.
We first recall from Section \eqref{sec.basic} that the operator $L_h$ defined in \eqref{APLiso}
 generates a $L^{2}$ semigroup that can be extended to $L^{q}$ for any $q\geq1$ by the technique of continuous extensions of bounded operators. Indeed, the functional $\mathcal{J}$ is a Dirichlet form on $L^{2}$ (see for instance \cite[Theorem 3.6, Theorem 4.1]{GrilloCipriani}). As a consequence, due to the fact that $L_{h}$ is positively homogeneous, for a given
nonnegative datum $u_{0}\in L^{1}(\R^{N})\cap L^{\infty}(\R^{N})$ and the
smoothing effect \eqref{Linfty-L1} we can apply
\cite[Theorem 1]{BC81} and  find for all $q\geq1$
\begin{equation}
\|\partial_{t}u\|_{q}\leq C\frac{\|u_{0}\|_{q}}{t}.\label{BCr}
\end{equation}
Then, if we take $u_{0}\in L^{1}(\R^{N})$, $u_{0}\geq0$, for any $\tau>0$
and $t\geq0$ we get
\begin{equation*}
\|\partial_{t}u(t+\tau)\|_{q}\leq C\frac{\|u(\tau)\|_{q}}{t},
\end{equation*}
thus if we combine this estimate with the smoothing effect \eqref{Linfty-L1} we obtain for all $t\geq\tau$
\begin{equation}
\|\partial_{t}u(t)\|_{\infty}\leq C\tau^{-\alpha-1}\|u_{0}\|^{{p}\alpha/N}\label{LIPCONST}.
\end{equation}
Hence equation \eqref{APL} can be viewed as the elliptic anisotropic equation
\begin{equation}
A_{h}(u):=-\sum_{i=1}^{N}\left(|u_{x_i}|^{p-2}u_{x_{i}}\right)_{x_{i}}=f\label{homoganis}
\end{equation}
where $f:=\partial_{t}u(\cdot,t)$ is a bounded source term. Then this equation fits into the Lipschitz regularity theory of \cite{CCG07}, whose main result implies what follows
\begin{theorem}\label{CELCCUP}
There exists and universal constant $C>0$ such that for all $u\in W^{1,1}
(B_{2R}(x_{0}))\cap L^{\infty}(B_{2R}(x_{0}))$ such that $A_{h}(u)= f$ weakly in $B_{2R}(x_{0})$, where $f\in L^{\infty}(B_{2R}(x_{0}))$, the following estimate holds:
\begin{equation}
\sup_{x\in B_{R}(x_{0})}|\nabla u|\leq C\left\{\int_{B_{2R}(x_{0})}\left[1+\frac{1}{p}\sum|\partial_{x_{i}} u|^{p}+\|f\|_{L^{\infty}} |u|\right]dx\right\}^{\alpha},\label{LIPSPACE}
\end{equation}
where $C=C(p,N,R,\|f\|_{L^{\infty}})$ and $\alpha=\alpha(p,N)$.
\end{theorem}

Then we are in position to prove the following result
\begin{theorem}\label{localLipcont}
Assume that $u_{0}\in L^{1}(\R^{N})$ and let $u$ be the mild solution to equation \eqref{APL}, satisfying the initial condition \eqref{IC}. Then, for all $\tau>0$, $u\in L^{\infty}(\R^{N}\times[\tau,+\infty))$ and $u$ is global Lipschitz continuous in $\R^{N}\times[\tau,\infty),$ with a bound
\begin{equation}\label{LIPSPACE2}
\sup_{\R^{N}\times[\tau,\infty)}|\nabla_{x,t}u(x,t)|\leq C(N,p,M,\tau,u_{0}).
\end{equation}
\end{theorem}

\begin{proof}
The fact that $u\in L^{\infty}(\R^{N}\times[\tau,+\infty))$ immediately follows from the $L^{1}$-$L^{\infty}$ smoothing effect \eqref{Linfty-L1}. Moreover, by estimate \eqref{LIPCONST} we have that $u$ is Lipschitz continuous in time for $t\geq \tau$. Finally, writing the parabolic equation as in \eqref{homoganis}, Theorem \ref{CELCCUP} yields global Lipschitz continuity in space: indeed, observe that using \eqref{BCr}, and \eqref{LIPCONST}, the Lipschitz estimate \eqref{LIPSPACE} implies (recall that $\nabla u(t)\in L^{p}(\ren))$ for any $t>0$ by Section \ref{sec.basic})
\[
|\nabla u(x_{0},t)|\leq C(N,p,M,\tau,u_{0})
\]
for all $x_{0}\in \ren$. Then $u$ is globally Lipschitz continuous in $\R^{N}\times[\tau,\infty)$.
\end{proof}
\begin{remark}
The local Lipschitz regularity in space in the range $p<2$ descends from the main result in \cite[Theorem 1]{PisanteVerde}. For the case $p>2$, gradient estimates for parabolic orthotropic equations has been recently established in \cite{BBLV21}.
\end{remark}


\subsection{Uniqueness of SSNI fundamental solutions
}\label{sec.uniq}

Now we give a stronger uniqueness result for nonnegative SSNI fundamental
solutions.

\begin{theorem}\label{thm.uniq.fundamental solution}
Let $p_i=p$ with $p_c<p<2$. The  nonnegative self-similar fundamental solution of equation \eqref{APL} with given mass $M>0$, given by
\[
B(x,t)=t^{-\alpha}F(t^{-\frac{\alpha}{N}}x)
\]
with the explicit profile $F$ of mass $M$ given by \eqref{Fp<2}, is the unique fundamental SSNI solution of mass $M$.
\end{theorem}

In particular the explicit self-similar fundamental solution \eqref{Fp<2}
 is the unique nonnegative fundamental SSNI solution of equation \eqref{APL} with given mass $M>0$.

\noindent {\sl Proof.} (i) By contradiction let us suppose there exist another SSNI fundamental solution $B_{1}$ to \eqref{APL}, with same mass $M$. We observe that $B_{1}$ satisfies the Lipschitz continuous stated in Theorem \ref{localLipcont}.

 We shall really need the non-degeneracy properties of $B$. A key point in the argument is that two different solutions with the same mass must intersect. We define the maximum of the two solutions $B^*=\max\{B_1,B\}$
and the minimum $B_*=\min\{B_1,B\}$. Obviously $B^*,$ and $B_*$ are positive and Lipschitz continuous solutions (w.r. to each variable) to \eqref{APL}. Under the assumption that the two functions $B_1$ and $B$ are not
the same, we define the open sets  $\Omega_{1}=\{(x,t)\in Q:\,B_1(x,t)<B(x,t)\}$ and { $\Omega_{2}=\{(x,t)\in Q:\,B_1(x,t)>B(x,t)\}$}, where as usual $Q=\R^{N}\times(0,\infty)$. Then $\Omega_{1}$ and $\Omega_{2}$ are disjoint and both are non-void open sets since the integrals of both functions over  $Q_{T}=\R^{N}\times(0,T)$ are the same for all $T>0$. In particular,  neither of them can be dense in $Q$. Moreover, $\Omega_{1}$ is the set where $B_*<B$ and $\Omega_{2}$ is the set where $B^*>B$.

(ii) We now show that the  situation $B_1\ne B$ is not possible because of strong maximum principle arguments applied to the difference of  the  two equations concerning $B^*$ and $B$. It is here that we use the fact that all the spatial derivatives of $B$ are different from zero away from
the set of points where a least one coordinate is zero, a set that we may
call the \sl coordinate skeleton\rm.  Its complement in $Q$ is given by
\[
\Omega=Q\setminus \bigcup_{i=1}^{N}\mathcal{A}_{i}
\]
where $\mathcal{A}_{i}=\left\{(x,t)\in Q:\,x_{i}=0\right\}$, for $i=1,...,N$, is an open set, the  union of  symmetric copies of $Q_i=\{(x,t)\in Q: \, x_i>0 \quad \forall i \}$. We will work in $\Omega$ to avoid the presence of degenerate points. We do as follows: we put  $$w(x,t)=B^*(x,t)-B(x,t),$$ then $w$ is nonnegative and continuous and satisfies (in the weak sense; recall that the stationary profiles are differentiable a.e.)
\begin{equation}\label{eq.ell.a}
w_{t}= \sum_i\left(a_i(x,t)w_{x_i}\right)_{x_i},
\end{equation}
The leading coefficients of the above equation are
\begin{equation}
a_i(x,t)=\frac{|B^*_{x_i}|^{p-2}B^*_{x_i}-|B_{x_i}|^{p-2}B_{x_i}}{B^*_{x_i}-B_{x_i}}\ge \frac{C_p}
{ |B^*_{x_i}|^{2-p}+|B_{x_i}|^{2-p}}>C_1>0,
\end{equation}
thus the locally Lipschitz continuity of the solutions given by Theorem \ref{localLipcont}, all  the $a_{i}(x,t)$ are locally bounded below by $C_1>0$. We see that each $a_{i}(x,t)$ is of the order of $\xi^{p-2}(x,t)$ for $\xi$ { between $|B^*_{x_i}|$ and $|B_{x_{i}}|$}. The problem is the bound from above, the equation might be not uniformly elliptic if we approach the skeleton.

(iii) Under our assumption $B_1\ne B$ we know that $w>0$ somewhere. By continuity we will have $w\ge c>0$ in a ball that does not intersect the skeleton, contained in a $Q_{i}$.  Then, $w$ cannot be zero everywhere in $\Omega$. Now, assume there is  point $P=(x,T)$ of intersection between $B^*$ and $B$, having all the coordinate values  nonzero, $x_i\ne 0$ for all $i$. Then $w(P)=0$. For definiteness, let us be in $Q_{1}$. In such a case $ |B_{x_i}|>c_i$ is bounded in a neighborhood of
$P$ for all $i$, and that means that all $a_i(x,t)$ are bounded above as announced in (ii). Indeed, arguing as in \cite[Lemma 5.1]{BobTak}, we can
write
\[
a_i(x,t)=\frac{|B_{x_i}|^{p-2}B_{x_i}-|B^*_{x_i}|^{p-2}B^*_{x_i}}{B_{x_i}-B^*_{x_i}}=(p-1)\int_{0}^{1}\left|s B_{x_i}+(1-s)B^*_{x_i}
\right|^{p-2}ds.
\]
We use the algebraic inequality
\[
\int_{0}^{1}|a+sb|^{p-2}ds\leq C_{p}\left(\max_{s\in[0,1]}|a+sb|\right)^{p-2},
\]
valid for all $a,b\in \R$ such that $|a|+|b|>0$, with the choice $a=B^*_{x_i}$ and $b=B_{x_i}-B^*_{x_i}$ (so that $|a|+|b|>c_{i}$ in the neighborhood of $P$), hence
\[
a_i(x,t)\leq C_{p}\left(\max_{s\in[0,1]}\left|s B_{x_i}+(1-s)B^*_{x_i}
\right|\right)^{p-2}\leq C.
\]
Considering the parabolic equation \eqref{eq.ell.a} in a small cylinder $Q_{\varepsilon,\tau,T}=B_\ve(x)\times(\tau,T)$,
the linear parabolic Harnack inequality (see \cite{Moser1,Moser2}) applies to it and we can conclude that necessarily $w$ must vanish identically in $Q_{\epsilon,\tau,T}$. By extension of the same principle $w$ must vanish in the whole $Q_{1}$, i.e. $B^*>B_{1}$ everywhere in $Q_{1}$.  What is important, this implies that $Q_{1}$ does not contain any point of $\Omega_{1}$. We now use the symmetry with respect to the axes and invariance
by translation with respect to any hyperplane $t=T$ and we arrive at the conclusion that $\Omega_{1}$ does not contain any interior point of any
quadrant. This is impossible.  \qed

\subsection{Asymptotic behaviour 
}\label{sec.asymp}

 In the orthotropic case, once the unique  SSNI  self-similar fundamental
solution $B_M$ is determined for any mass $M>0$, it is natural to expect that this is the good candidate to be the attractor for solutions to the Cauchy problem for equation \eqref{APL}. Indeed, we have the following result:

\begin{theorem}
Let $p_i=p$ for all $i$ with $p_c<p<2$. Let $u(x,t)\ge 0$ be the unique
weak solution of the Cauchy problem of the orthotropic equation \eqref{APL} (i.e., $p_i=p$)  with initial data $u_{0}\in L^{1}(\R^{N})$ of mass $M$. Let $B_{M}$ the Barenblatt solution
\begin{equation}\label{BM}
B_{M}(x,t)=t^{-\alpha}F(t^{-\frac{\alpha}{N}}x)
\end{equation}
with $F$ defined in \eqref{Fp<2} having mass $M$. Then,
\begin{equation}\label{conv.l1}
\lim_{t\rightarrow\infty}\|u(t)-B_{M}(t)\|_{1}=0.
\end{equation}
The convergence holds in the $L^{\infty}$ norm in the proper scale
\begin{equation}\label{conv.linfty}
\lim_{t\rightarrow\infty}t^{\alpha}\|u(t)-B_{M}(t)\|_{\infty}=0.
\end{equation}
where $\alpha$ is given by \eqref{alfa}. Weighted convergence in $L^q(\ren)$, $1<q<\infty$ is obtained by interpolation.
\end{theorem}

\begin{proof}
First, let us observe that the smoothing effect estimate \eqref{Linfty-L1} implies in particular that $u(t)\in L^{2}(\R^{N})$ for all $t\geq\tau$,
for any $\tau>0$, so that $u$ is the solution $\eqref{APL}$ for $t\geq\tau$ with datum in $L^{2}(\R^{N})$. It follows from the theory that  $u$ is
a \emph{strong semigroup} $L^{2}$ solution, as explained in
Section \ref{sec.basic}, meaning that the first and the second energy estimate \eqref{firstenergy}, \eqref{secondenergy} hold in any time interval
$(\tau,T)$. Let us define now the family of rescaled solutions. For  all $\lambda>0$ we put
\[
u_{\lambda}(x,t)=\lambda^{\alpha}u(\lambda^{\frac{\alpha}{N}}x,\lambda t).
\]
By the mass invariance it follows that, for all $\lambda>0$,
\[
\|u_{\lambda}(\cdot,t)\|_{1}=M=\|u(\cdot,t)\|_{1}
\]
and the smoothing estimate \eqref{Linfty-L1} yields
 for any $\bar t>0$
\begin{equation}\label{lambda1}
\|u_{\lambda}(\cdot,{\bar t})\|_{\infty}=\lambda^{\alpha}\|u_{\lambda}(\cdot,\lambda {\bar t})\|_{\infty}\leq C {\bar t}^{-\alpha} M^{p\alpha/N}.
\end{equation}
Then since the norms $\|u_{\lambda}(\cdot,{\bar t})\|_{1}$
and $\|u_{\lambda}(\cdot,\bar t)\|_{\infty}$ are equi-bounded w.r. to $\lambda$, we have by interpolation that the norms $\|u_{\lambda}(\cdot,{\bar t})\|_{p}$ are equi-bounded for all $p\in [1,\infty]$. Now we fix $\color{magenta}\bar t>0$, so that by the previous remark $u({\bar t})\in L^{2}(\R^{N})$ and we can use the first energy
estimate \eqref{firstenergy} for $t\geq {\bar t}$:
\[
\sum_{i=1}^{N}\int_{{\bar t}}^{t}\int_{\R^{N}}|u_{x_{i}}|^{ p}\,dx\,d\tau\leq \frac{1}{2}\|u({\bar t})\|_{2}^{2}.
\]
Moreover, \eqref{esttimeder1} and \eqref{esttimeder2} provide
\[
\int_{\bar t}^{t}\int_{\R^{N}}|u_{t}(x,\tau)|^{2}dx\,d\tau\leq C\frac{\|u({\bar t})\|_{2}^{2}}{t}.
\]
Then we have
\begin{equation}\label{lambda2}
\sum_{i=1}^{N}\int_{{\bar t}}^{t}\int_{\R^{N}}|\partial_{x_{i}}u_{\lambda}|^{ p}\,dx\,d\tau\leq
C\lambda^{\alpha}\|u(\cdot,\lambda {\bar t})\|_{2}^{2}=C\|u_{\lambda}(\cdot,{\bar t})\|_{2}^{2}
\end{equation}
and since $\|u_{\lambda}(\cdot,{\bar t})\|_{2}$ is equibounded, we have that $\partial_{x_{i}}u_{\lambda}$ are equibounded in $L^{p}_{x,t}$ for $i=1,\cdots,N$, $t\geq {\bar t}$. Moreover,
we have the following estimate of the time derivatives:
\begin{align}
\int_{{\bar t}}^{t}\int_{\ren}|\partial_{t}u_{\lambda}(x,\tau)|^{2}dx\,d\tau&=\lambda^{\alpha+1}\int_{\lambda {\bar t}}^{\lambda t}\int_{\R^{N}}|\partial_{t} u(x,\tau)|^{2}dx\,d\tau\leq C\lambda^{\alpha}\frac{\|u(\cdot,\lambda {\bar t})\|_{2}^{2}}{{\bar t}}\nonumber\\
&=\frac{C}{{\bar t}} \|u_{\lambda}(\cdot,\lambda {\bar t})\|_{2}^{2}, \label{lambda3}
\end{align}
and this gives weak compactness of the time derivatives $\partial_{t} u_{\lambda}$ in $L^{2}_{x,t}$ for $t\geq {\bar t}$. Then estimates \eqref{lambda1} \eqref{lambda2} and \eqref{lambda3} imply, for $t\geq {\bar t}$: $u_{\lambda}\in L^{\infty}_{x,t}$, $\partial_{x_{i}}u_{\lambda}\in L^{p}_{x,t}$ for every $i$, $\partial_{t}u_{\lambda}\in L^{2}_{x,t}$ with uniform bounds w.r. to $\lambda$. Then Rellich-Kondrachov Theorem allows to say that the family $u_{\lambda}$ is relatively locally compact in $L^{1}_{x,t}$. Therefore, up to subsequences, we have
\begin{equation}
\lim_{\lambda\rightarrow\infty}u_{\lambda}(x,t)=U(x,t)
\end{equation}
for some finite-mass function $U(x,t)\ge 0$, and the convergence holds in
$L^{1}_{loc}(Q)$. Then arguing as in \cite[Lemma 18.3]{Vlibro}, it is easy to show that $U$ is a \emph{weak} solution to \eqref{APL}, in the sense
that
\[
\int_{t_{1}}^{t_{2}}\int_{\R^{N}}U\,\varphi_{t}dx\,dt
-\sum_{i=1}^{N}\int_{t_{1}}^{t_{2}}\int_{\R^{N}}|\partial_{x_i}U|^{p-2}\partial_{x_i}U\partial_{x_{i}}\varphi dx\,dt=0
\]
for all the test functions $\varphi\in C_{0}^{\infty}(\R^{N}\times (0,\infty))$.

\smallskip

(ii) Assuming that $u_{0}$ is bounded and compactly supported in a ball $B_{R}$, we argue as in \cite[Theorem 18.1]{Vlibro}.  We take a larger mass $M^{\prime}>M$ and the self-similar solution $B_{M^{\prime}}(x,t)$ such
that $B_{M^{\prime}}(x,1)\geq u_{0}(x)$. Then we clearly have
\[
u_{\lambda}(x,0)=\lambda^{\alpha}u(\lambda^{\frac{\alpha}{N}}x,0)\leq
\lambda^{\alpha} B_{M^{\prime}}(\lambda^{\frac{\alpha}{N}}x,1)=B_{M^{\prime}}\left(x,\frac{1}{\lambda}\right).
\]
Then the comparison principle gives
\begin{equation}\label{rescaledcompar}
u_{\lambda}(x,t)\leq B_{M^{\prime}}\left(x,t+\frac{1}{\lambda}\right).
\end{equation}
Since $u_{\lambda}\rightarrow U$ a.e. and $B_{M^{\prime}}\left(x,t+\frac{1}{\lambda}\right)\rightarrow  B_{M^{\prime}}(x,t)$  as $\lambda\rightarrow\infty$, the mass invariance of $B_{M^{\prime}}$ and \eqref{rescaledcompar} allows to apply Lebesgue dominated convergence Theorem and obtain (up to subsequence)
\[
u_{\lambda}(t)\rightarrow U(t)\quad \text{in}\, L^{1}(\R^{N})
\]
which means that the mass of $U$ is equal to $M$ at any positive time $t$. This gives that $U$ is a fundamental solution with initial mass $M$, it is bounded for all $t>0$ and the usual estimates apply.
Moreover, observe that the rescaled sequence $u_{\lambda}$ have initial data supported in a sequence of shrinking balls $B_{R/\lambda^{\frac{\alpha}{N}}}(0)$. The usual application of the Aleksandrov Principle implies that $U(x,t)$ will have the properties of monotonicity along coordinate directions and also the property of symmetry with respect
to coordinate hyperplanes. For more details, see \cite[Theorem 3]{KamVaz}. Then the uniqueness Theorem \eqref{thm.uniq.fundamental solution} applies and we have $U=B_{M}$. Actually we have that any subsequence of $u_{\lambda}(t)$ converges in $L^{1}(\ren)$ to $B_{M}(t)$, thus the whole family of rescaled solutions $u_{\lambda}(t)$ converges to $B_{M}(t)$ in $L^{1}(\ren)$.\\ In particular we have $u_{\lambda}(x,1)\rightarrow B_{M}(x,1)=F(x)$ in $L^{1}(\R^{N})$  with $F$ defined in \eqref{Fp<2}\nlc, which gives formula \eqref{conv.l1}.  The general case $u_{0}\in L^{1}(\R^{N})$ can be done by following the arguments in \cite[Theorem 18.1]{Vlibro}.

\smallskip

(iv) Now we pass to achieve the uniform convergence \eqref{conv.linfty}. First of all, the equiboundedness of the family $u_{\lambda}$ and the Lipschitz estimates \eqref{LIPSPACE} given by Theorem \ref{localLipcont} allow the use of the Ascoli-Arzel\'a Theorem, in order to obtain
\[
u_{\lambda}\rightarrow B_{M}
\]
uniformly on compact sets of $Q=\R^{N}\times (0,\infty)$. In order to obtain the full convergence in $\R^{N}$ at time $t=1$ we need a tail analysis at infinity and we argue as in \cite[Theorem 18.1]{Vlibro}. Take any $\varepsilon>0$, then the very definition of the rescaled solutions $u_{\lambda}$ gives, for $\lambda>1$ and $R>1$,
\begin{align*}
\int_{|x|>R/2}u_{\lambda}(x,1)dx&=\int_{|x|>R/2}\left[u_{\lambda}(x,1)-F(x)\right]dx+\int_{|x|>R/2}F(x)dx\\
&\leq \int_{\R^{N}}\left[u(y,\lambda)-B_{M}(y,\lambda)\right]dx+\int_{|x|>R/2}F(x)dx
\end{align*}
Now, \eqref{conv.l1} allows to select a sufficiently large $\lambda$ such
that
\[
\int_{\R^{N}}\left|u(y,\lambda)-B_{M}(y,\lambda)\right|dy<\frac{\varepsilon}{2},
\]
Then choosing a large $R>>1$ such that
\[
\int_{|x|>R/2}F(x)dx<\frac{\varepsilon}{2}
\]
we have for $\lambda$ large
\begin{equation}
\int_{| x|>R/2}u_{\lambda}(x,1)dx<\varepsilon.\label{cola1}
\end{equation}
Let us take any $x_{0}$ such that $|x_0|>R$, so that $B_{R/2}(x_{0})\subset\left\{|x|>R/2\right\}$. From the Gagliardo-Nirenberg inequality on bounded domains (see \emph{e.g.} \cite{Nire} or \cite{Gagliardo}) we have
\[
\|u_{\lambda}(\cdot,1)\|_{L^{\infty}( B_{R/2}(x_{0}))}\leq C_{1}\|u_{\lambda}(\cdot,1)\|_{L^{1}( B_{R/2}(x_{0}))}^{\widetilde{\alpha}}
\|\nabla u_{\lambda}(\cdot,1)\|^{1-\widetilde{\alpha}}_{L^{\infty}( B_{R/2}(x_{0}))}+C_{2}\|\|u_{\lambda}(\cdot,1)\|_{L^{1}( B_{R/2}(x_{0}))},
\]
where $\widetilde{\alpha}=1/(N+1)$ and $C_{i}$, $i=1,2$ are constants depending on $N$, $x_{0}$ and $R$.
Then, by \eqref{cola1} and  the uniform bound of the gradient \eqref{LIPSPACE2} we have, for $\lambda$ large,
\[
\|u_{\lambda}(x,1)\|_{L^{\infty}(B_{R/2}(x_{0}))}\leq C\varepsilon^{\alpha},
\]
therefore for all $x_{0}$ such that $|x_{0}|>R$,
\[
u_{\lambda}(x_{0},1)\leq C\varepsilon^{\alpha}.
\]
Thus the uniform convergence on compact sets implies that $u_{\lambda}(x,1)\rightarrow F(x)$ uniformly on $\R^{N}$, as $\lambda\rightarrow\infty$,
which easily translates to \eqref{conv.linfty}.
\end{proof}
%
\section{Complements on the theory}\label{sec.comp}

\subsection{A comparison theorem}

  First we prove a comparison for solutions to a Cauchy-Dirichlet problem
associated to equation \eqref{APL} posed on a domain $U$, where $U$ can be bounded or unbounded (in the latter case we will consider $U$ either as
an outer domain (\emph{i.e.} the complement of a bounded domain) or a half space.
Let us consider the following Cauchy-Dirichlet problem
\begin{equation}\label{APL U}
\left\{
\begin{array}
[c]{ll}%
u_t=\sum_{i=1}^N\left(|u_{x_i}|^{p_i-2}u_{x_i}\right)_{x_i} & \text{ in } U\times [0,\infty)\\
& \\
u(x,t)=h(x,t)\geq0& \text{ in } \partial U\times [0,\infty)\\
& \\
u(x,0)=u_0(x)\geq 0& \text{ in } U,
\end{array}
\right.
\end{equation}
where in general we  take $u_0\in L^1(U)$ and $h\in C(\partial U\times [0,\infty))$.
\begin{proposition} \label{Prop ComU}
Suppose that $u_{1}$ and $u_{2}$ are two positive smooth solutions of \eqref{APL U} with initial data $u_{0,1},\,u_{0,2}\in L^1(U)$ and boundary data $h_1\leq h_2$ on $\partial U\times [0,\infty)$.  Then we have
\begin{equation}\label{contractprincU}
\int_{U} (u_1(t)-u_2(t))_+\,dx\le \int_{U} (u_{0,1}-u_{0,2})_+\,dx\,.
\end{equation}
In particular, if $u_{0,1}\le u_{0,2}$ for a.e.  $x\in U$, then for every
$t>0$ we have $u_1(t)\le u_2(t)$ a.e. in $U$.
\end{proposition}
\begin{proof} We point out that the boundary conditions of  $u_{1}, u_{2}$  on $\partial U$ implies in particular that $u_{1}\leq u_{2}$ on $\partial U$ implies in particular $(u_{1}-u_{2})_{+}=0$ on $\partial U$. We follow the lines of the proof of \eqref{elliptL1_contr} in Theorem \ref{ellipticexi}. Indeed using the same test function we find
\begin{equation*}
\begin{split}
\frac{d}{dt}\int_{U}&(u_{1}(t)-u_{2}(t))^{+}\,\zeta_{n}(x)dx
\\
&
=\sum_{i=1}^N\int_{U}\partial_{x_{i}}
\left(|\partial_{x_{i}}u_1|^{p_i-2}\partial_{x_{i}}u_1-|\partial_{x_{i}}u_2|^{p_i-2}\partial_{x_{i}}u_2\right)
(u_{1}-u_{2})_{+}\,\zeta_{n}(x)dx\\
&=-\sum_{i=1}^N
\int_{U}
\left(|\partial_{x_{i}}u_1|^{p_i-2}\partial_{x_{i}}u_1-|\partial_{x_{i}}u_2|^{p_i-2}\partial_{x_{i}}u_2\right)
(u_{1}-u_{2})_{+}\,\partial_{x_{i}}\zeta_{n}(x)dx
\\
&+\sum_{i=1}^N\int_{\partial{U}} \left(|\partial_{x_{i}}u_1|^{p_i-2}\partial_{x_{i}}u_1-|\partial_{x_{i}}u_2|^{p_i-2}\partial_{x_{i}}u_2\right)
(u_{1}-u_{2})_{+}\zeta_{n}(x)\nu_i\, d \sigma
\\
&
={-}\sum_{i=1}^N\int_{U}
\left(|\partial_{x_{i}}u_1|^{p_i-2}\partial_{x_{i}}u_1-|\partial_{x_{i}}u_2|^{p_i-2}\partial_{x_{i}}u_2\right)
(u_{1}-u_{2})_{+}\,\partial_{x_{i}}\zeta_{n}(x)dx.
\end{split}
\end{equation*}
From now on we argue as in i) in the proof of Theorem \ref{ellipticexi}.
\end{proof}

\subsection{Aleksandrov's reflection principle}
This is an auxiliary section used in the proof of Aleksandrov's principle
so we will skip unneeded generality.
Let $H_j^+=\{x\in \mathbb{R}^N:x_j>0\}$  be the positive half-space with respect to the $x_j$ coordinate for any fixed $j\in\{1,\cdots,N\}$.  For any $j=1,\cdots,N$  the hyperplane $H_j=\{x_j=0\}$  divides $\mathbb{R}^N$ into two half spaces $H_j^+=\{x_j>0\}$ and $H_j^-=\{x_j<0\}$. We denote by $\pi_{H_j}$ the specular symmetry that maps a point $x\in
H_j^+$  into $\pi_{H_j}(x)\in H_j^-$, its symmetric image with respect to
$H_j$. We have the following important result:

\begin{proposition}\label{Prop Al}
Let $u$ a positive solution of the Cauchy problem for \eqref{APL} with
positive initial data $u_0\in L^1(\mathbb{R}^N)$.
If  for a given hyperplane $H_j$ with  $j=1,\cdots,N$ we have
$$u_0(\pi_{H_j}(x))\leq u_0(x)\, \text{ for { a.e.} }x\in H{_j}^{+}$$
then for all $t$
$$u(\pi_{H_j}(x),t)\leq u(x,t)\quad \text{ for { a.e.} }(x,t)\in H_{j}^{+} \times  (0,\infty).$$
\end{proposition}

\begin{proposition}\label{Prop 3}
Let $u$ be a positive solution of the Cauchy problem for \eqref{APL} with
nonnegative
initial data $u_0\in L^1(\mathbb{R}^N)$. If $u_0$ is a  symmetric function in each variable $x_i$,  and also a decreasing function in $|x_i|$ for all $i$ a.e., then $u(x,t)$  is also symmetric and a nonincreasing function in $|x_i|$ for all $i$ for all $t$ a.e. in $x$ (for short  SSNI, meaning separately symmetric and nonincreasing).
\end{proposition}


In order to prove the previous two propositions we can argue as in \cite{FVV2020}. In particular Proposition \ref{Prop Al} is a consequence to Proposition \ref{Prop ComU} and yields Proposition \ref{Prop 3}.

\section{Control on the anisotropy}\label{sec.aniso}

\begin{figure}[t!]
 \centering
 \vspace{-3.0cm}
 \includegraphics[scale=0.7]{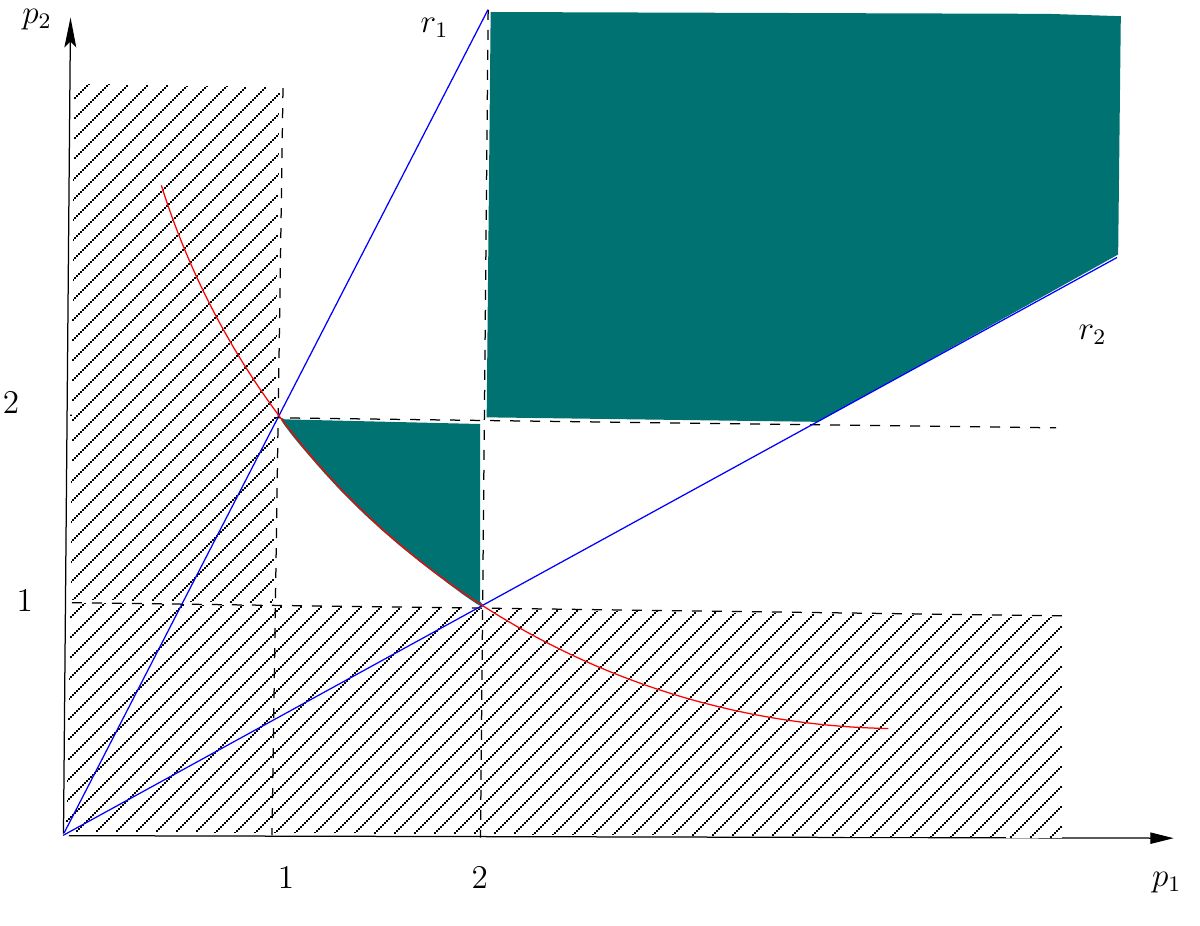}
 \caption{$p_1,p_2$ that verifies conditions (H2)-(H3) when $p_1,p_2\leq2$ or $p_1,p_2\geq2$}
  \label{fig:1}
\end{figure}

\noindent $\bullet$ In our analysis of existence of self-similar solution
for equation (APLE) we have found conditions (H2) and (H3). It is interesting to examine what these requirements mean for $N=2$ and $p_1,p_2>1$.
Then,
 condition (H2) means
$$
\frac{p_1p_2}{p_1+p_2} >\frac{2}{3}\quad \text{ i.e. } (p_1-2/3)(p_2-2/3)>4/9.
$$
The region is limited below in Figure  \ref{fig:1} by a symmetric hyperbola in which passes through the points $(2,1)$, $(4/3,4/3)$, and $(1,2)$. As for condition (H3), we have
$$
p_i<3/2\overline{p}= 3p_1p_2/(p_1+p_2),
$$
which amounts to $p_1<2p_2$ (delimited by line $r_2$ in the Figure \ref{fig:1})  and symmetrically $p_2<2p_1$ (delimited by line $r_1$). We thus get a necessary ``small anisotropy condition'' which takes the form
$$
\frac{1}2 <\frac{p_1}{p_2}<2
$$
and it is automatically satisfied for fast diffusion $1<p_1,p_2<2$.

\noindent $\bullet$ The analysis of the (APME) in \cite{FVV2020} leads to
a simpler algebra. According to the results of the paper, condition (H2) becomes
$$
\frac1{N}\sum_i m_i> \frac{N-2}{N},
$$
which in dimension $N=2$ reads $m_1+m_2>0$. For $N\ge 3$ we get $m_1+m_2+\cdots m_N>N-2$. This is much simpler than the (APLE) condition. Otherwise the anisotropy control (H3) reads
$$
m_i< \overline{m}+\frac2{N}
$$
which for $N=2$ means 
$$
|m_1-m_2|<2.
$$
This is automatically satisfied for fast diffusion $0<m_1,m_2<1$, but is important when slow diffusion occurs in some coordinate direction.

\section{Self-similarity for Anisotropic Doubly Nonlinear Equations }\label{sec.dne}

We have studied two types of anisotropic evolution equations: the anisotropic  equation of Porous Medium type (APME) treated in \cite{FVV2020}, and the model (APLE) involving anisotropic $p$-Laplacian type \eqref{APL}, studied here above. The similarities lead to consider a more general evolution equation with anisotropic nonlinearities involving powers of both the solution and its spatial derivatives
\begin{equation}\label{APL:dnl}
u_t=\sum_{i=1}^N\left(\left|(u^{m_i})_{x_i}\right|^{p_i-2}(u^{m_i})_{x_i}\right)_{x_i}.
\end{equation}
We will call it (ADNLE). We assume that $m_i>0$ and $p_i>1$. The isotropic case is well known, see Section $11$ of  \cite{VazSmooth}. We describe next the self-similarity analysis applied to solutions plus the physical requirement of finite conserved mass.

\medskip

\noindent $\bullet$ The type of self-similar solutions of equation \eqref{APL} has again the usual form
$$
B(x,t)=t^{-\alpha}F(t^{-a_1}x_1,..,t^{-a_N}x_N)
$$
with constants $\alpha>0$, $a_1,..,a_n\ge 0$ to be chosen below. We substitute this formula into equation \eqref{APL:dnl}.
\noindent Note that  writing $y=(y_i)$ with $y_i=x_i \,t^{-a_i}$,
equation \eqref{APL:dnl} becomes
$$
-t^{-\alpha-1}\left[\alpha F(y)+
\sum_{i=1}^{N} a_iy_i\,F_{y_i}
\right]=\sum_{i=1}^{N}t^{-[\alpha m_i(p_i-1)+p_ia_i]}\left(|(F^{m_i})_{y_i}|^{p_i-2}(F^{m_i})_{y_i}\right)_{y_i}.
$$
Time is eliminated as a factor in the resulting equation on the condition
that:
\begin{equation}\label{ab:dnl}
\alpha(m_i(p_i-1)-1)+p_ia_i=1 \quad \mbox{ for all } i=1,2,\cdots,  N.
\end{equation}

\medskip

\noindent $\bullet$ We also look for integrable solutions that will enjoy
the mass conservation property, and this
implies that \ $\alpha=\sum_{i=1}^{N}a_i$. Writing $a_i=\sigma_i \alpha,$ we get the conditions $\sum_{i=1}^{N}\sigma_i=1$ and
$$
\alpha\left[m_i(p_i-1)-1 + p_i\sigma_i\right]=1 \qquad \forall i.
$$
From this set of conditions we can get the unique admissible values of $\alpha$ and $\sigma_i$.
We proceed as follows. From the last displayed  formula we get
\begin{equation}\label{sigmai}
\sigma_i=\frac{1}{p_i}\left(\frac{1}{\alpha}+1-m_i(p_i-1)\right).
\end{equation}
Then  condition $\sum_{i=1}^{N}\sigma_i=1$ implies that
$$
1= \left(\frac{1}{\alpha}+1\right)
\sum_{i=1}^{N}\frac{1}{p_i}-\sum_{i=1}^N m_i +\sum_{i=1}^N \frac{m_i}{p_i}\,.
$$
At this moment we introduce some suitable notations:
$$
\frac{1}{N}\sum_{i=1}^{N}\frac{1}{p_i}={\overline p}, \quad \frac{1}{N}\sum_{i=1}^{N}{m_i}={\overline m}, \quad
\frac{1}{N}\sum_{i=1}^{N}\frac{m_i}{p_i}=\frac{q}{\overline p}\,.
$$
Using them, we get
\begin{equation}\label{alfa:dnl}
\alpha=\frac{N}{N({\overline m}\,\overline{p}-q-1)+\overline{p}}\,.
\end{equation}
We want to work in a parameter range that  ensures that $\alpha>0$, and this means condition
\begin{equation}\label{Cpos:dnl}\tag{DN2}
{\overline p}\,\,{\overline m}+ \frac{\overline p}{N}> q+1,
\end{equation}
which is the equivalence in this setting to condition (H2) in the Introduction. Under this condition the self-similar solution will decay in time in maximum value like a power of time. This is a crucial condition for the self-similar solution to exist and play its role, since the suitable existence theory contains the maximum principle.

\medskip

\noindent $\bullet$
Once $\alpha $ is obtained, the $\sigma_i$ are given by \eqref{sigmai}. These exponents  control the rate of spatial spread in every coordinate direction, we know that \ $\sum_{i=1}^{N}\sigma_i=1$, and in particular
$\sigma_i=1/N$ in the homogeneous case. The condition to ensure that $\sigma_i>0$ is 
\begin{equation}\label{VV}\tag{DN3}
m_i(p_i-1)<\frac{1}{\alpha}+1= {\overline p}\,\,{\overline m}+ \frac{\overline p}{N}- q.
\end{equation}
This means that the self-similar solution expands  as time passes (or at least it does not contract), along any of the coordinate directions.

Note that the simple fast diffusion conditions $m_i<1$ and $p_i<2$ and $\alpha>0$ ensure that $\sigma_i> 0$.

\medskip

\noindent \bf 2. Particular cases\rm. (1) When all the $m_i=1$ we find the results of our present paper contained in Section \ref{sec.sss} for equation (APLE). On the other hand, when $p_i=2$ we find the results of the previous paper \cite{FVV2020} for equation (APME).

(2) It is also interesting to look at cases where the $m_i=m$ but not necessarily 1, and when $p_i=p$ but not necessarily 2. In the first case
$q=m$, while in the second case we get $q=\overline{m}$. In both cases $\alpha$ is given by the simpler formula
$$
\alpha=\frac{N}{N({\overline m} \,(\overline{p}-1)-1)+\overline{p}}\,,
$$
that looks very much like the isotropic case, see the Barenblatt solution
which is explicitly written in Subsection 11.4.2 of  \cite{VazSmooth}.

\medskip

\noindent \bf 3. On the theory\rm.
With these choices,  the profile function $F(y)$ must satisfy the following doubly-nonlinear anisotropic stationary equation in $\mathbb{R}^N$:
\begin{equation*}
\sum_{i=1}^{N}\left[\left(|(F^{m_i})_{y_i}|^{p_i-2}(F^{m_i})_{y_i}\right)_{y_i}+\alpha\sigma_i\left( y_i F\right)_{y_i}\right]=0.
\end{equation*}
Conservation of mass must also hold : $\int B(x,t)\, dx =\int F(y)\, dy=M<\infty $ \ {\rm for} $t>0.$

The next  step would be to prove that there exists a suitable solution of
this elliptic equation, which is the anisotropic version of the equation of the Doubly nonlinear Barenblatt profiles in the standard $m$-$p$-Laplacian.  The solution is indeed explicit in the isotropic case, as we have said.

\section{Comments, extensions and open problems}

\begin{itemize}
\item We may replace the main equation \eqref{APL} by
\begin{equation*}
u_t=\sum_{i=1}^N \left(a_i|u_{x_i}|^{p_i-2}u_{x_i}\right)_{x_i}\quad\text{ in }\quad Q:=\mathbb{R}^N\times(0,+\infty)
\end{equation*}
 with all constants $a_i>0$ and nothing changes in the theory. Inserting the constants may be needed in the applications. The case where the $a_i$
depend on $x$ appears in inhomogeneous media and it is out of our scope. And we did not touch on the theory of equations like \eqref{APL} where the exponents $p(x,t)$ are space-time dependent, see in this respect \cite{AntSh15}.

\item We may replace the main equation \eqref{APL} by
\begin{equation*}
 u_t=\sum_{i=1}^N  (|u_{x_i}|^{p_i-2}u_{x_i})_{x_i}+ \ve \Delta_p(u) \quad \mbox{in  } \ \quad Q:=\mathbb{R}^N\times(0,+\infty).
 \end{equation*}

  At least in the case of homogeneous anisotropy  the same theory will work and we have uniqueness of self-similar solutions, which are also  explicit and we can write them.

\item The cases where some or all of the $p_i$ are larger than 2 are not treated here in any systematic way. Notice that our general theory applies, as well as the symmetrization and boundedness. The upper barrier has to changed into a barrier compatible with the compact support properties. In the orthotropic case, the existence theorem for self-similar Barenblatt solutions obtained in paper \cite{CV2020} can be completed with the proof of uniqueness and the theorem of asymptotic behaviour as in our Section \ref{sec.ortho} above.

\item  The limit cases where some $p_i=2$ deserve attention.

\item  Symmetrization does not give sharp bounds probably when the $p_i$ are not the same,
but it implies the $L^\infty$ bounded where the constant is explicit. Can
we compare our self-similar solutions with the isotropic Barenblatt by symmetrization?

\item If we check the explicit self-similar solutions of the isotropic and orthotropic equations, they are comparable but for a constant.

\item  We have not discussed the Harnack or the H\"older regularity for this theory.

\item  Following the idea of \cite{N} it si possible to prove a strong maximum principle in the \emph{homogeneous case} where all exponents are equal, $p_{1}=...p_{N}=p<2$.

\begin{theorem}
Let $T>0$, $\Omega$ a domain of $\mathbb{R}^N$, $u\in C^0([0,T)\times\Omega)$ satisfying $u_t-L_hu\geq 0$ with $L_h$ defined as \eqref{APLiso}, $p<2$ and data $u_0$ non-identically zero such that $u(\cdot,t)\geq0$ on $\partial\Omega$, $\forall t\geq0$. If there exists some $x\in\Omega$ and  $t>0$ such that $u(x, t)=0$, then $u(\cdot,t)\equiv0$ on $\Omega$.
\end{theorem}
\end{itemize}

\vskip-1cm
\section*{Acknowledgments}
J. L. V\'azquez was funded by  grant PGC2018-098440-B-I00 from MICINN (the Spanish Government). J.~L.~V\'azquez is an Honorary Professor at Univ. \ Complutense de Madrid. The authors wish to warmly thank L. Brasco for fruitful discussions and valuable suggestions.


\

\noindent {\bf 2020 Mathematics Subject Classification.}
  	35K55,  	
   	35K65,   	
    35A08,   	
    35B40.   	

\medskip

\noindent {\bf Keywords: } {Nonlinear parabolic equations, $p$-Laplace diffusion, anisotropic equation, symmetrization, fundamental solutions,  asymptotic behaviour}.

\end{document}